\documentclass[a4paper]{article}

\usepackage[all]{xy}\usepackage[latin1]{inputenc}        
\usepackage[dvips]{graphics,graphicx}
\usepackage{amsfonts,amssymb,amsmath,color,mathrsfs, amstext}
\usepackage{amsbsy, amsopn, amscd, amsxtra, amsthm,authblk,enumerate,algorithmicx,algorithm}
\usepackage{algpseudocode}
\usepackage{upref}
\usepackage{geometry}
\usepackage[displaymath]{lineno}
\usepackage{float}
\usepackage{yhmath}
\usepackage[colorlinks,
            linkcolor=red,
            anchorcolor=red,
            citecolor=red
            ]{hyperref}

\numberwithin{equation}{section}

\def\epsilon{\varepsilon}

\DeclareMathOperator{\var}{var}

\newtheorem{theorem}{Theorem}[section]
\newtheorem{lemma}{Lemma}[section]

\newtheorem*{proposition*}{Proposition}
\newtheorem{corollary}{Corollary}[section]
\newtheorem*{corollary*}{Corollary}

\newtheorem*{definitions*}{Definitions}
\newtheorem*{conjecture*}{\bf Conjecture}

\newtheorem*{example*}{\bf Example}
\theoremstyle{remark}
\newtheorem{remark}{\bf Remark}[section]

\newtheorem{assumption}{Assumption}[section]

\newcounter{parentalgorithm}
\newenvironment{subalgorithms}{%
  \refstepcounter{algorithm}%
  \setcounter{parentalgorithm}{\value{algorithm}}%
  \setcounter{algorithm}{0}%
  \ignorespaces
}{%
  \setcounter{algorithm}{\value{parentalgorithm}}%
  \ignorespacesafterend
}
\makeatother

\numberwithin{equation}{section}

\begin{document}

\title{Random Batch Methods (RBM) for interacting particle systems}

\author[1]{Shi Jin \thanks{shijin-m@sjtu.edu.cn}}
\author[2]{Lei Li\thanks{leili2010@sjtu.edu.cn}}
\author[3]{Jian-Guo Liu\thanks{jliu@phy.duke.edu}}
\affil[1,2]{School of Mathematical Sciences, Institute of Natural Sciences, MOE-LSC, Shanghai Jiao Tong University, Shanghai, 200240, P. R. China.}
\affil[3]{Department of Mathematics and Department of Physics, Duke University, Durham, NC 27708, USA.}

\date{}
\maketitle

\begin{abstract}
We develop Random Batch Methods for interacting particle systems with large number of particles. These methods use small but random batches for particle interactions, thus the computational cost is reduced from $O(N^2)$ per time step to $O(N)$, for a system with $N$ particles with binary interactions.  On one hand, these methods are efficient Asymptotic-Preserving schemes for the underlying particle systems, allowing $N$-independent time steps and also capture, in the $N \to \infty$ limit, the solution of the mean field limit which are nonlinear Fokker-Planck equations; on the other hand, the stochastic processes generated by the algorithms can also be regarded as new models for the underlying problems. For one of the methods, we give a particle number independent error estimate under some special interactions. Then, we apply these methods to some representative problems in mathematics, physics, social and data sciences, including the Dyson Brownian motion from random matrix theory, Thomson's problem, distribution of wealth, opinion dynamics and clustering. Numerical results show that the methods can capture both the transient solutions and  the global equilibrium in these problems.
\end{abstract}

\section{Introduction}

In natural and social sciences, there are many collective behaviors resulted from a huge number of interrelated individuals. Examples include swarming or synchronization described by the Vicsek model \cite{vicsek1995novel}, flocking in school of fishes \cite{degond2017coagulation}, groups of birds \cite{cucker2007emergent}, chemotaxis of bacteria \cite{horstmann03}, consensus clusters in opinion dynamics \cite{motsch2014}, to name a few. 

We are interested to develop random algorithms for these interacting particle systems to reduce computational cost. 
While the ideas also work for second order systems like Langevin dynamics \cite{schlick2010molecular} and the model in \cite{cucker2007emergent},  let us firstly focus on the following first order systems:
\begin{gather}\label{eq:interactingps0}
d X^i=-\nabla V(X^i)\,dt+\sum_{k: k\le J}\alpha_{N,k}\sum_{\mathscr{C}\subset\{1,\ldots, N\}: i\in \mathscr{C}, |\mathscr{C}|=k} K^i(\mathscr{C})\,dt+\sigma dB^i,~i=1,\ldots, N,
\end{gather}
where $J$ is independent of $N$ and $\alpha_{N,k}$ are some constants. In other words, the interaction acting on each particle is the superposition of the ones for which the number of particles involved is independent of $N$. Here, $X^i\in\mathbb{R}^d$ are the labels for particles, $-\nabla V(\cdot)$ is some given external field, $K^i(\cdot)$'s are the interaction kernels and $\{B^i\}$'s are independent standard Brownian motions.  Such systems could be the overdamped limit of Langevin equations, where $dX^i$ comes from the friction term so that $X^i$ is the location of the $i$th particle and the terms on the right hand side are forces. Of course, $X^i$ can also have other interpretations depending on applications, like the velocity \cite{cucker2007emergent}, the opinion or wealth (see section \ref{sec:app}). We will loosely call $X^i$'s the locations of particles, $-\nabla V$ the external force and $K(\cdot)$ the interacting forces for convenience in this paper, unless there are clear meanings. If $\sigma=0$, there is no diffusion effect. If $\sigma>0$, we have diffusion, and the equation is a stochastic differential equation (SDE) in It\^o's sense \cite{oksendal03}. We will mostly focus on $J=2$, so that the systems are of binary interactions (i.e. each particle interacts with others separately). System \eqref{eq:interactingps0} in the case of binary interaction can be written as
\begin{gather}\label{eq:interactingps}
d X^i=-\nabla V(X^i)\,dt+\frac{1}{N-1}\sum_{j: j\neq i} K(X^i-X^j)\,dt+\sigma dB^i,~i=1,\ldots, N.
\end{gather}
We will present our algorithms and analysis for binary interactions (system \eqref{eq:interactingps}) and consider that the initial values $X_i(0)=:X_0^i$ are drawn independently from the same distribution:
\begin{gather}
X_0^i \sim \nu,~i=1,\ldots, N.
\end{gather}
 Extensions to \eqref{eq:interactingps0} are straightforward, though more involved (see section \ref{sec:diss} and Remark \ref{rmk:extensiontoq}).
Although there are examples of many body interactions like Kac interaction in spin glasses \cite{frohlich1987some,franz2004finite} for \eqref{eq:interactingps0}, we remark that the binary interaction is much more important and ubiquitous in nature, like Coulomb's interaction between electrons and nuclei, due to the superposition principles for forces.  As another remark, we have assumed the additive noise for simplicity to illustrate our methodology. However, the discussion in this paper also applies for multiplicative noise (i.e. $\sigma$ depends on $x$). (We in fact have one such example in Section \ref{sec:wealth}.)

There may be two things we care about regarding \eqref{eq:interactingps0} and \eqref{eq:interactingps}, depending on applications. The first thing is the dynamics, for example, in the opinion dynamics, we would like to know how the consensus of opinions are developed.  Another thing is the equilibrium distribution of the $N$ particles. If $N$ is large, simulation of \eqref{eq:interactingps0} and \eqref{eq:interactingps} is very expensive since for each time step during the evolution, the computational cost is of $O(N^J)$.  One possible method for studying large and complex systems where individuals interact with each other is the mean field approximation \cite{stanley1971, georges1996, lasry2007}. In this approximation, the effect of surrounding particles is approximated by a consistent averaged force field so that one obtains a one body problem.
For example, regarding \eqref{eq:interactingps} that we will focus on, in the $N\to \infty$ regime, the distribution $\mu$ of the particles formally satisfies the following nonlinear Fokker-Planck equation (see \cite{kac1956,mckean1967,benachour1998,liuyang2016,jabinquantitative,serfaty2018mean,li2018})
\begin{gather}\label{eq:nonlinearfp}
\partial_t\mu=\nabla\cdot(\nabla V(x)\mu)-\nabla\cdot(\mu K*\mu)+\frac{1}{2}\sigma^2\Delta \mu.
\end{gather}
It is expected that $\mu$ is close to the empirical measure for \eqref{eq:interactingps}
\begin{gather}\label{eq:randomempirical}
\mu_N(t):=\frac{1}{N}\sum_{i=1}^N \delta(x-X^i(t)),
\end{gather}
which is a random measure. In fact, under certain assumptions on $V$ and $K$, it can be shown that the uniform mean field limit holds (\cite{cattiaux2008}).
Both the marginal distribution of $X^1$ in \eqref{eq:interactingps}, denoted as $\mu_N^{(1)}$, and \eqref{eq:randomempirical} are close to $\mu$.

On the other hand, one often cares more about the mean field equations like \eqref{eq:nonlinearfp} and its invariant measure $\pi$, but the mean field equation and $\pi$ are hard or expensive to solve or compute. Using the mean field approximation, one can generate some artificial particle systems of the form \eqref{eq:interactingps0} or \eqref{eq:interactingps}. In this sense, the interacting particle systems \eqref{eq:interactingps0} and \eqref{eq:interactingps} are Monte-Carlo particle methods for the mean field equation and $\pi$. Clearly, when $N$ (the number of numerical particles) is large, one still wants to reduce the computational cost.
Hence, no matter whether studying system \eqref{eq:interactingps} (and \eqref{eq:interactingps0} ) is due to its own interest or due to numerical simulation of the mean field equation, it is highly desirable to design some efficient algorithms to solve the particle system \eqref{eq:interactingps}. 

Nowadays, in the era of big data, many stochastic algorithms have been developed to reduce the computational cost while keep certain accuracy.  Hence, one may borrow some ideas from these areas for the physical problems. The stochastic gradient descent (SGD) algorithm is developed to reduce the computational cost and for better exploring the high dimensional parameter space \cite{bottou1998online,bubeck2015convex} for supervised learning \cite{schmidhuber2015}. In SGD, a small batch of samples are chosen each time to form the noisy gradient. A similar algorithm is the stochastic coordinate descent where only a few coordinates are updated each time \cite{nesterov2012efficiency,wright2015coordinate}.  The idea of random batches also appears in the stochastic gradient Langevin dynamics (a Markov Chain Monte Carlo method (MCMC)), which can be applied for Bayesian inference \cite{welling2011bayesian,ma2015complete,nagapetyan2017true}. Pretty much like SGD, the method uses random batch of data to update the parameters one needs to estimate. Besides, many random algorithms also appear for matrices, such as constructing approximate matrix decompositions by randomized Singular Value Decompositions \cite{halko2011} and the fast Monte Carlo algorithm for matrix-matrix product in low rank problems \cite{drineas2006fast1,drineas2006fast2}. 

Motivated by the stochastic algorithms in machine learning and statistics, we develop Random Batch Methods (RBM) for interacting particle systems \eqref{eq:interactingps} and \eqref{eq:interactingps0}. The idea is quite simple: for a {\it small} duration of time, one randomly divides the particles into {\it small} batches and the interactions are only turned on inside each batch (see Section \ref{sec:alg} for details). Numerical experiments in later sections show that RBMs work quite well for a wide range of applications with binary interactions from physical, biological to data sciences.  These algorithms reduce the computational cost per time step from $O(N^2)$ to $O(N)$ for binary interactions \eqref{eq:interactingps} ($O(N^J)$ to $O(N)$ for \eqref{eq:interactingps0}).  They not only recover the equilibrium well in variaous problems, but also approximate {\it dynamics} of measures with very singular interaction kernels (see for example section \ref{sec:dyson}). Moreover in section \ref{sec:error}, under suitable conditions on the external potential, interaction force and batch size, we prove for one of the RBMs that the numerical error has a bound that only depends on the time step,  {\it independent} of $N$ and time. The key for the proof is Lemma \ref{lmm:averagefunc}, which guarantees that on average the random force is consistent with the full interaction.

The methodology of randomly choosing some objects has also been widely used in kinetic theory community. 
The multi-level Monte Carlo (MLMC) method \cite{haji2018multilevel} runs simulations with different number of particles for McKean-Vlasov equation using the idea of subsampling. The authors used smaller numbers of particles for repeated simulations with a subset of random variables (such as the Brownian motions) for the finest simulation to reduce the computational cost, while keeping good enough MLMC estimator. Different from \cite{haji2018multilevel}, RBMs use the idea of subsampling in computing interacting forces for one single simulation to approximate the distribution of particles.
The Direct Simulation Monte Carlo (DSMC) method proposed by Bird (\cite{bird1963approach,bird1994molecular}) uses randomly chosen simulation molecules to solve the Boltzmann equation with finite Knudsen number (see the work by Wagner \cite{wagner1992convergence} for the proof of convergence).  Moreover, Nanbu's simulation method directly derived from the Boltzmann equation has gained success and been proved to converge \cite{nanbu1980direct,babovsky1989convergence}.
In \cite{carlen2013kinetic}, biological swarm models with random pair interactions have been studied and proven to converge to Kac-Boltzmann equation. In the reverse way to \cite{carlen2013kinetic},  the work by Albi and Pareschi (\cite{albi2013}) uses stochastic binary interaction algorithms for flocking and swarming dynamics. The algorithms in \cite{albi2013} are intrinsically doing similar things as our RBMs do; their Algorithm 4.3 is particularly alike RBM-1 (see section \ref{sec:alg}) with $p=2$. See a more recent application by Carrillo et al \cite{carrillo2017particle}. In this sense, our RBMs are generalizations of symmetric Nanbu (Algorithm 4.3) in \cite{albi2013}.
However, the numerical particle system in \cite{albi2013} was motivated by the mean-field limit, in a way similar to the DSMC methods of binary collisions for the Boltzmann equation.   It was aimed at solving the mean-field equation (corresponding to the right vertical line in Fig. \ref{fig:convergencetauN}).  RBMs are motivated by the idea of mini-batch methods from machine learning. They are Monte-Carlo type approximations directly to the (physical) particle systems \eqref{eq:interactingps}. They  correspond to the left horizontal line in Fig. \ref{fig:convergencetauN}. Moreover, one of the RBMs will be proved (under special conditions) to converge, when time steps go to zero, to the particle system \eqref{eq:interactingps}.  The construction of RBMs, as well as the proof of their convergence,  are obtained  {\it without the knowledge of the mean-field limit}.  Of course when $N\to \infty$ its density measure also converges to the mean-field limit equation, as will be proven using the Wasserstein distance. Therefore, the RBMs are a class of Asymptotic-Preserving schemes for particle systems \eqref{eq:interactingps} or \eqref{eq:interactingps0}, in the sense of 
\cite{jin2010asymptotic} (see Fig. \ref{fig:convergencetauN}). Lastly, the idea of turning on interactions inside the batches can be easily extended. For example, one can change the batch size to adjust the noise levels. Moreover, RBMs can also be extended to \eqref{eq:interactingps0} involving more complex interactions.

Let us remark that RBMs can be straightforwardly applied to the second order particle systems. In Appendix \ref{sec:Hamil}, we establish similar error estimates for a Hamiltonian system, valid in finite time, and perform the numerical verification in section \ref{subsec:1dsimpletest}.

The rest of the paper is organized as follows. In section \ref{sec:alg}, we propose the Random Batch Method without replacement (RBM-1) and with replacement (RBM-r), and give some discussions. In section \ref{sec:error}, we obtain an error estimate, which only depends on time step but not on $N$ and time,  of the marginals for RBM-1 in the Wasserstein distance under some assumptions for the external and interacting forces. Though performed for binary interactions, the analysis should work for \eqref{eq:interactingps0} as well. In section \ref{sec:numverify}, we do numerical tests to verify the theoretic results in section \ref{sec:error} and use the Dyson Brownian motion to compare the RBMs.  In particular, in section \ref{sec:dyson}, we compute the law of the Dyson Brownian motions from random matrix theory where the interaction kernel is {\it singular}. Our method can capture the evolution of distribution and the equilibrium semicircle law, and the two algorithms (RBM-1 and RBM-r) give comparable results. In section \ref{sec:app}, we use the RBMs to solve some interesting application problems with binary interactions.  In section \ref{sec:thomson}, we focus on Thomson's problems where we solve the dynamics on the sphere.  In section \ref{sec:social}, we design stochastic dynamics for evolution of wealth and opinions. Lastly, we apply our randomized algorithms for efficient clusters in section \ref{sec:cluster}. The paper is concluded in section \ref{sec:conclusion}.

\section{The Random Batch Methods}\label{sec:alg}

In this section, we propose the RBMs by using random batches for the summation of the interacting force in \eqref{eq:interactingps}. The extensions to \eqref{eq:interactingps0} should be similar but are more involved, and we will give some discussion about this in section \ref{sec:diss}.
For the setup, we pick a short duration of time $\tau$ and consider the discrete time
\begin{gather}
t_m=m\tau.
\end{gather}
Suppose we compute up to time $T$ and the number of iteration for the stochastic algorithm is
\begin{gather}\label{eq:NT}
N_T=\left\lceil \frac{T}{\tau} \right\rceil.
\end{gather}
Clearly to simulate the ODE system \eqref{eq:interactingps} directly, the complexity is 
$O(N_T N^2)$. If $N$ is large, this is expensive. Motivated by the stochastic algorithms in the machine learning community, we will use a randomized strategy.

\subsection{The first algorithm (RBM-1)}\label{sec:firstrandom}

A natural idea is that at each time $t_m$, we divide the $N=np$ particles into $n$ small batches with size $p$ ($p\ll N$, often $p=2$) randomly, denoted by $\mathcal{C}_q, q=1,\ldots, n$, and then interact particles within each batch (For convenience, we have assumed that $p$ divides $N$. In general, the last batch does not have to have size $p$.)
The algorithm is called RBM-1 (shown in \ref{randomdiv}). Each iteration contains two main steps: (1) Randomly shuffling and dividing the particles into $n$ batches; (2) evolving with interactions only turned on inside batches.

\begin{algorithm}[H]
\caption{(RBM-1)}\label{randomdiv}
\begin{algorithmic}[1]
\For {$m \text{ in } 1: [T/\tau]$}   
\State Divide $\{1, 2, \ldots, pn\}$ into $n$ batches randomly.
     \For {each batch  $\mathcal{C}_q$} 
     \State Update $X^i$'s ($i\in \mathcal{C}_q$) by solving the following SDE with $t\in [t_{m-1}, t_m)$.
     \begin{gather}\label{eq:firstalgorithm}
            dX^i=-\nabla V(X^i) dt+\frac{1}{p-1}\sum_{j\in \mathcal{C}_q,j\neq i}K(X^i-X^j)dt+\sigma dB^i.
      \end{gather}
      \EndFor
\EndFor
\end{algorithmic}
\end{algorithm}

Clearly, the update equation \eqref{eq:firstalgorithm} can be rewritten as
\begin{gather}
dX^i=-\nabla V(X^i) dt+\frac{1}{N-1}\sum_{j: j\neq i}K(X^j-X^i)dt+\sigma dB^i
+\chi_{m,i}(X(t))\,dt,
\end{gather}
where
\begin{gather}
\chi_{m,i}(X(t))=\frac{1}{p-1}\sum_{j\in \mathcal{C}_q,j\neq i}K(X^i-X^j)-\frac{1}{N-1}\sum_{j: j\neq i}K(X^i-X^j).
\end{gather}

For a given $x=(x^1,\ldots,x^N)\in\mathbb{R}^{Nd}$ that is independent of the random division, we have (see Lemma \ref{lmm:averagefunc} for the proof)
\begin{gather}
\mathbb{E}\chi_{m,i}(x)=0.
\end{gather}
This is a key observation which eventually leads to the convergence of the algorithms in expectation. 
As a remark, the position $X=(X^1,\ldots, X^N)$ itself depends on the random division. Hence, in general,
\[
\mathbb{E}\chi_{m,i}(X(t))\neq 0.
\]

Regarding the complexity, note that random division into $n$ batches of equal size can be implemented using random permutation. The latter can be realized in $O(N)$ operations by Durstenfeld's modern revision of Fisher-Yates shuffle algorithm \cite{durstenfeld1964}. (In MATLAB, one can use 'randperm(N)' to generate a random permutation. Then, the first $p$ elements are considered to be in the first batch, the second $p$ elements are in the second batch, etc). 

If one is to simulate up to time $T$, the total number of time steps is $N_T$ as in \eqref{eq:NT}. 
The cost for each batch per iteration is $\lesssim Cp^2$, where $C$ is the unit time independent of $N,p,\tau$. Since there are $n=N/p$ batches, so the total cost per iteration is like $\lesssim CpN$. The total complexity is therefore
\begin{gather*}
O(pN_TN)=O\left(\frac{p}{\tau}NT\right).
\end{gather*} 
The ration $p/\tau$ affects the total complexity. If $p=2$ which we will use for the numerical simulations due to the simplicity in implementations, the complexity is $O(N_TN)=O(TN/\tau)$. The cost is reduced significantly from the naive simulation.

In section \ref{sec:error}, one will see that the RBM converges due to some averaging effect in time, and 
the rate follows in a similar way as Law of Large Number results. The error bound is like $C\sqrt{\tau/p}$. This means the ratio $\tau/p$ controls the error theoretically if the ODEs/SDEs can be solved accurately. Hence, if one wants the error to be smaller than some given accuracy $\epsilon$, the computational cost is like $O(NT/\epsilon^2)$.
The dependence in $N$ is linear! The power for $\epsilon$ is $2$, which is typical for Monte Carlo methods. To improve this, one may consider some variance reduction techniques (see for example \cite{fang2018spider}).

RBM-1 is in spirit similar to the stochastic gradient descent in machine learning (\cite{bottou1998online,bubeck2015convex}).
Recently, there are some analysis of SGD in the mathematical viewpoints and applications to physical problems \cite{litaie2017,feng2017,hulililiu2018,chen2018online}.
RBM-1 can be used both for simulating the evolution of the measure \eqref{eq:randomempirical} or \eqref{eq:nonlinearfp} and for sampling from the equilibrium state $\pi$.

\subsection{The Random Batch Method with replacement}\label{sec:secondrandom}
RBM-1 requires the random division, and the elements in different batches are different.
This is in fact the sampling without replacement. If one allows replacement, we have RBM-r$'$.

\begin{subalgorithms}
\begin{algorithm}[H]
\caption{(RBM-r$'$)}\label{randomreplacement}
\begin{algorithmic}[1]
\For {$m \text{ in } 1: [T/\tau]$}   
     \For {$k$ from $1$ to $N/p$}
     \State Pick a set $\mathcal{C}_k$ of size $p$ randomly with replacement. 
     \State Update $X^i$'s ($i\in \mathcal{C}_k$) by solving the following SDE for time $\tau$.
     \begin{gather}\label{eq:algorithmreplacement}
     \left\{
           \begin{split}
            & dY^i =-\nabla V(Y^i) dt+\frac{1}{p-1}\sum_{j\in \mathcal{C}_k,j\neq i}K(Y^i-Y^j)dt+\sigma dB^i,\\
            & Y^i(0) =X^i,
            \end{split}
            \right.
      \end{gather}
 \quad\quad\quad i.e., solve \eqref{eq:algorithmreplacement} with initial values $Y^i(0)=X^i$, and set $X^i\leftarrow Y^i(\tau)$.
      \EndFor
\EndFor
\end{algorithmic}
\end{algorithm}

RBM-r$'$ can be reformulated as Random Batch Method with replacement (RBM-r) that has some flavor of  the stochastic coordinate descent method \cite{nesterov2012efficiency,wright2015coordinate}. Here, the pseudo-time $s$ is introduced for convenience and $s_m=m\tau$. Roughly, $t_m$ corresponds to $s_{mN/p}$.

\begin{algorithm}[H]
\caption{(RBM-r)}\label{fullyrandom}
\begin{algorithmic}[1]
\For {$m \text{ in } 1: [T/\tau]*(N/p)$}   
\State Pick a set $\mathcal{C}$ of size $p$ randomly.
     \State Update $X^i$'s ($i\in \mathcal{C}$) by solving the following with pseudo-time $s\in [s_{m-1}, s_m)$.
     \begin{gather}\label{eq:secondalgB}
            dX^i=-\nabla V(X^i)\, ds+\frac{1}{p-1}\sum_{j\in\mathcal{C},j\neq i}K(X^i-X^j)\,ds+\sigma\, dB^i.
      \end{gather}
 \EndFor
\end{algorithmic}
\end{algorithm}
\end{subalgorithms}

For such type of methods, we expect to have 
\begin{gather}
\tilde{N}_T=\frac{1}{p}N N_T
\end{gather}
iterations to get comparable behaviors. However, each step is very cheap: one only needs $O(p^2)$ work. Hence, we still expect the complexity to be $O(pNN_T)$.

Though the same as RBM-r$'$, 
RBM-r does not have explicit concept of time, since the positions of the particles are not changing simultaneously. Intuitively,  RBM-r can only give sampling for the invariant measure for the nonlinear Fokker-Planck equation \eqref{eq:interactingps} while unable to simulate the dynamics of the distribution.  However, as RBM-r$'$ indicates,  $N/p$ iterations correspond to time $\tau$ and thus a single step in RBM-1, so it might still approximate the evolution of distributions. In fact, the example later confirms this and $m(N/p)$ iterations indeed give acceptable approximation for the distribution at time $m\tau$.
As a last comment, we may sometimes choose random time step; for example $\tau\sim Exp(\Delta t)$ (the exponential distribution with parameter $\Delta t$) such that $\mathbb{E}\tau=\Delta t$. Intuitively, this may help to increase the noise level and avoid being trapped in some local minimizers of $V$.

\subsection{Some discussions about the algorithms}\label{sec:diss}

In this subsection we make some discussion about our algorithms and complementary remarks. 
\begin{enumerate}
\item For system \eqref{eq:interactingps0}, the RBMs can be similarly developed. The batch size should be larger than or equal to $J$ ($p\ge J$). After the random batch is chosen, the interactions only happen inside the randomly chosen batch. One should also adjust the magnitude of the interaction such that an analogy of Lemma \ref{lmm:averagefunc} holds. In other words, the expectation of the random forces should equal to those in \eqref{eq:interactingps0}. The complexity clearly is reduced from $O(N^J)$ to $O(N)$.

\item The sizes of batch do not have to be equal. One can even choose the sizes randomly.
One can also adjust the batch size to adjust the noise level. 

\item After the batches are obtained, one can do parallel computing to update the positions of $X$ for RBM-1. However, one has to do regrouping after one time interval which stops the parallelism. There is no big efficiency improvement unless solving the SDEs is expensive. Regarding complexity,  there is no big difference in complexity. For RBM-r, to get $p$ indices, one needs to generate $p$ random numbers from $[0, 1]$, and the total complexity is $O(N)$ for $N/p$ iterations. For RBM-1,  as in Durstenfld's implementation of random permutation, one needs to generate $N$ random numbers from $[0, 1]$ while do some swapping operations for each element. (This may be different from the intuition established from drawing balls from a bag. For randomly dividing $N$  balls, one can choose $p$ from $N$ balls first, then $p$ from $N-p$, $\ldots$. It seems that the complexity becomes smaller and smaller as the total number of balls is becoming smaller and smaller. This is indeed not the case because in computer one can use arrays to store the 'balls' and one only needs to generate the indices without touching these 'balls'.) 

In summary, there is no big difference in complexity for the two algorithms. The advantage of RBM-1 might be its ability for parallelism during one interval while the advantage of RBM-r is its simplicity so that it is likely more flexible for extensions. 

\item In Lemma \ref{lmm:averagefunc}, for $x$ independent of the random division, we prove that
\[
\mathrm{Var}(\chi_{m,i}(x))=\left(\frac{1}{p-1}-\frac{1}{N-1}\right)\Lambda_i(x),
\]
where $\Lambda_i(x)$ is independent of $p$.
This means for larger $p$, the variance is smaller and the noise level is lower. This noise somehow reflects the fluctuation of the empirical measure $\tilde{\mu}_N^{(1)}$ around $\mu_N^{(1)}$.

\item If $K$ is a singular forcing (like Coulomb), we can do {\it splitting} method and have
\begin{gather}
\begin{split}
&dX^i= \frac{1}{p-1}\sum_{j\in \mathcal{C},j\neq i} K(X^i-X^j)\,dt,\\
&dX^i=-\nabla V(X^i)\,dt+\sigma dB^i.
\end{split}
\end{gather}
If $p=2$, the singular forcing term can be often solved {\it analytically}. This is another advantage of the stochastic algorithm: for the $N$-particle system, if the forcing $K$ is singular so that the problem is stiff, explicit scheme needs very small time step while implicit scheme is hard to invert. Using the stochastic algorithm plus time splitting, the evolution can be solved {\it exactly}, thus avoiding stability constraint.

\item Suppose one aims at the law of particles in \eqref{eq:interactingps}
with very large $N$. By mean field limit, one can choose to
solve \eqref{eq:nonlinearfp} directly. A possible way is to use particle method for \eqref{eq:nonlinearfp}  where the number of particles is much smaller than $N$ and the masses for the particles in the numerical method can even be different (the particle blob method \cite{liu2017random}).  For the interacting particle system of this numerical purpose, our randomized algorithms also apply well. In this regard, we provided an efficient numerical particle method for the mean-field (or kinetic) equations. In fact this is exactly the starting point of the binary algorithms in \cite{albi2013}.

\item 
To better simulate the invariant measure in the case of $\sigma=0$, one may add the Brownian motion $\epsilon_N dB^i$ where $\epsilon_N$ decreases with $N$. For example, we may set 
\[
\epsilon_N=\frac{1}{N^{\gamma}},~\gamma>0.
\]

\item 
If the initial distribution is far from the equilibrium and we aim to get the global equilibrium, one may need many iterations for convergence to the equilibrium. A possible way is to use the Gibbs distribution corresponding to $V$ 
\[
\nu(dx)\propto \exp(-V(x))\,dx
\]
 for initialization to reduce the number of iterations. This distribution can be sampled using the Markov chain Monte Carlo (MCMC) methods \cite{hastings1970monte,gilks1995markov}. Of course, for special $V$ such as quadratic functions, the initial distribution can be sampled directly without MCMC.
 
 \item Compared with traditional treecode and fast multipole method \cite{rokhlin1985rapid} where the interaction kernels are required to decay, the RBMs work also for kernels that do not necessarily decay to reduce the computational cost to $O(N)$. They can also be applied for kernels of the general form $K(x,y)$ which may not be symmetric and may not have translational invariance.
\end{enumerate}

\section{An error analysis for RBM-1}\label{sec:error}

In this section, we study RBM-1 proposed in Section \ref{sec:firstrandom}. In particular, we will check how close it is to the fully coupled system \eqref{eq:interactingps}. We leave the study of RBM-r to the future. However, as commented above, we expect RBM-r to also work when RBM-1 works. Recall that each iteraction consists of random division and evolution. The mechanism that makes RBMs work is Lemma \ref{lmm:averagefunc} and small step size. The philosophy is as following. Suppose that the system has certain chaotic property. When $\tau$ is small enough, the accumulative behavior along many time steps will be roughly comparable to the average behavior, which is \eqref{eq:interactingps} thanks to Lemma \ref{lmm:averagefunc}. This is similar to Law of Large Numbers, but on time.

We assume the following conditions on the confining and interacting potentials:
\begin{assumption}\label{ass:basicassumption}
Suppose $V$ is strongly convex on $\mathbb{R}^d$ so that $x\mapsto V(x)-\frac{r}{2}|x|^2$ is convex, and $\nabla V$, $\nabla^2V$ have polynomial growth (i.e. $|\nabla V(x)|+|\nabla^2 V(x)|\le C(1+|x|^q)$ for some $q>0$).
Assume $K(\cdot)$ is bounded, Lipschitz on $\mathbb{R}^d$ with Lipschitz constant $L$ and has bounded second order derivatives. Moreover,
\begin{gather}
r>2L.
\end{gather}
\end{assumption}

The condition $r>2L$ is to ensure that the evolution group for the deterministic part of \eqref{eq:interactingps} is a contraction. In particular, if 
\[
\dot{X}^i=-\nabla V(X^i)+\frac{1}{N-1}\sum_{j: j\neq i}K(X^i-X^j),~~\dot{Y}^i=-\nabla V(Y^i)+\frac{1}{N-1}\sum_{j: j\neq i}K(Y^i-Y^j),
\]
 then
\begin{gather}
\frac{d}{dt}\sum_{i=1}^N|X^i-Y^i|^2\le -(r-2L) \sum_{i=1}^N|X^i-Y^i|^2.
\end{gather}
For such systems, one needs $\sigma>0$ for a nontrivial equilibrium; otherwise, all the particles will go to a single point.

Our goal is to prove that the distribution generated by RBM-1 is close to the 
distribution of the marginal distribution of the $N$ particle system in the Wasserstein distance. We recall that the Wasserstein-$2$ distance is given by \cite{santambrogio2015}
\begin{gather}\label{eq:W2}
W_2(\mu, \nu)=\left(\inf_{\gamma \in \Pi(\mu,\nu)}\int_{\mathbb{R}^d\times\mathbb{R}^d}|x-y|^2 d\gamma\right)^{1/2},
\end{gather}
where $\Pi(\mu,\nu)$ means all the joint distributions whose marginal distributions are $\mu$ and $\nu$ respectively.

In order to achieve this goal, we consider the synchronous coupling between \eqref{eq:interactingps} and \eqref{eq:firstalgorithm}. 
We denote $X^i$ the solution obtained by the $N$ interacting particle system while $\tilde{X}^i$ the solution obtained by RBM-1. Correspondingly, $\tilde{B}^i$ will denote the Brownian motion used in RBM-1. Note that both \eqref{eq:interactingps} and \eqref{eq:firstalgorithm} have exchangeability. This means, for example, the joint distribution of $X^i$'s is symmetric. Consequently, $\tilde{X}^i$ and $\tilde{X}^j$ (for all $i$ and $j$) have the same distribution.  We construct the coupling as follows. For $i=1,2,\ldots, N$,
\begin{gather}\label{eq:coupling}
\begin{split}
& \tilde{X}^i(0)=X^i(0)= X_0^i \sim \nu,\\
& B^i(t)=\tilde{B}^i(t).
\end{split}
\end{gather}
With this coupling, we define the error process
\begin{gather}
Z^i(t):=\tilde{X}^i(t)-X^i(t).
\end{gather}

Let $\xi_{m-1}$ denote the random division of batches at $t_{m-1}$, and  define the filtration $\{\mathcal{F}_{m}\}_{m\ge 0}$ by
\begin{gather}
\mathcal{F}_{m-1}=\sigma(X_0^i, B^i(t), \xi_j; t\le t_{m-1}, j\le m-1).
\end{gather}
In other words, $\mathcal{F}_{m-1}$ is the $\sigma$-algebra generated by the initial values $X_0^i$ ($i=1,\ldots, N$), $B^i(t)$, $t\le t_{m-1}$, and $\xi_j$, $j\le m-1$. Hence, $\mathcal{F}_{m-1}$ contains the information of how batches are constructed for $t\in [t_{m-1}, t_m)$. We also introduce the filtration $\{\mathcal{G}_m\}_{m\ge 0}$ by
\begin{gather}
\mathcal{G}_{m-1}=\sigma(X_0^i, B^i(t), \xi_j; t\le t_{m-1}, j\le m-2).
\end{gather}
If we use $\sigma(\xi_{m-1})$ to mean the $\sigma$-algebra generated by $\xi_{m-1}$, the random division of batches at $t_{m-1}$, then $\mathcal{F}_{m-1}=\sigma(\mathcal{G}_{m-1}\cup \sigma(\xi_{m-1}))$. Throughout this section, $C$ will denote generic constants whose concrete value can change from line to line. We use $\|\cdot\|$ to represent the $L^2(\Omega)$ norm ($\Omega$ means the sample space):
\begin{gather}
\|v\|=\sqrt{\mathbb{E}|v|^2}.
\end{gather}

We now state the theorem.
\begin{theorem}\label{thm:mainresult}
Suppose Assumption \ref{ass:basicassumption} holds. With the coupling constructed above,
\begin{gather}
\sup_{t\ge 0}\|Z^1(t)\|\le C\sqrt{\frac{\tau}{p-1}+\tau^2},
\end{gather}
where $C$ is independent of $N, p$ and $t$. Consequently, let $\mu_N^{(1)}(t)$ be the first marginal for \eqref{eq:interactingps} and $\tilde{\mu}^{(1)}_N$ be the first marginal for system \eqref{eq:firstalgorithm}. Then
\begin{gather}
\sup_{t\ge 0}W_2(\mu^{(1)}_N(t), \tilde{\mu}^{(1)}_N(t))\le C\sqrt{\frac{\tau}{p-1}+\tau^2}\le  C\sqrt{\tau}.
\end{gather}
\end{theorem}

We need some preparation for the proof.  Lemma \ref{lmm:averagefunc} is a type of consistency lemma for RBMs while Lemma \ref{lmm:momcond} somehow shows the stability of RBM-1.
\begin{lemma}\label{lmm:averagefunc}
Consider $p\ge 2$ and a given fixed $x=(x^1, \ldots, x^N)\in \mathbb{R}^{Nd}$. Then, for all $i$,
\begin{gather}\label{eq:averageofzetami}
\mathbb{E}\chi_{m,i}(x)=0,
\end{gather}
where the expectation is taken with respect to the random division of batches. Moreover, the variance is given by
\begin{gather}
\var(\chi_{m,i}(x))=\mathbb{E}|\chi_{m,i}(x)|^2=\left(\frac{1}{p-1}-\frac{1}{N-1}\right)\Lambda_i(x),
\end{gather}
where
\begin{gather}
\Lambda_i(x) :=\frac{1}{N-2}
\sum_{j: j\neq i}\left| K(x^i-x^j)-\frac{1}{N-1}\sum_{k:k\neq i}K(x^i-x^k)\right|^2.
\end{gather}
\end{lemma}
\begin{proof}
We use $I(i,j)=1$ to indicate that $i,j$ are in the same batch. 
We rewrite
\[
f_{m,i}:=\frac{1}{p-1}\sum_{j:\in \mathcal{C}_q,j\neq i}K(x^i-x^j)
=\frac{1}{p-1}\sum_{j: j\neq i} K(x^i-x^j)1_{I(i,j)=1}.
\]
We first show that
\[
\mathbb{E}1_{I(i,j)=1}=\frac{p-1}{N-1}.
\]

As is well-known, there are 
\[
M(n):=\frac{(pn)!}{(p!)^n n!}
\]
ways of dividing $pn$ distinguishable objects into $n$ batches of size $p$. 
For a given $(i, j)$ pair, to compute $\mathbb{E}1_{I(i,j)=1}$ (the probability that $i, j$ are in the same batch), we compute the number of ways to group $(i, j)$ together. We first choose $p-2$ objects from $np-2$ to group with $i,j$ and then form a batch. Then, divide the remaining into $n-1$ batches.   Hence,
the number of ways to make $i,j$ in the same batch is ${np-2 \choose p-2}M(n-1)$ and thus
\[
\mathbb{E}1_{I(i,j)=1}
=\frac{{np-2 \choose p-2}M(n-1)}{M(n)}=\frac{p-1}{N-1}.
\]
This then proves that
\[
\mathbb{E}f_{m,i}=\frac{1}{N-1}\sum_{j:j\neq i}K(x^i-x^j),
\]
and thus \eqref{eq:averageofzetami} follows.

Now, let us consider the variance. We first compute the second moment of $\xi_{m,i}$:
\begin{gather*}
\begin{split}
\mathbb{E}|f_{m,i}|^2
=&\frac{1}{(p-1)^2}\sum_{j:j\neq i}|K(x^i-x^j)|^2 \mathbb{P}(I(i,j)=1)\\
 &+\frac{1}{(p-1)^2}\sum_{j,k: j\neq k, j\neq i,k\neq i}K(x^i-x^j)\cdot K(x^i-x^k) \mathbb{P}(I(i,j)I(i,k)=1).
\end{split}
\end{gather*}
By similar argument,
\[
\mathbb{P}( I(i,j)I(i,k)=1)=\frac{{np-3 \choose p-3}M(n-1)}{M(n)}
=\frac{(p-1)(p-2)}{(N-1)(N-2)}.
\]
Hence,
\begin{gather*}
\begin{split}
\mathbb{E}|\chi_{m,i}(x)|^2
=&~\mathbb{E}|f_{m,i}|^2-(\mathbb{E}|f_{m,i}|)^2 \\
=&\left(\frac{1}{p-1}-\frac{1}{N-1} \right)\Big(\frac{1}{N-1}\sum_{j\neq i}|K(x^i-x^j)|^2 \\
&-\frac{1}{(N-1)(N-2)}\sum_{i,j: j\neq k, j\neq i, k\neq i}K(x^i-x^j)\cdot K(x^i-x^k)\Big) \\
=&\left(\frac{1}{p-1}-\frac{1}{N-1} \right)\frac{1}{N-2}
\sum_{j: j\neq i}\left|K(x^i-x^j)-\frac{1}{N-1}\sum_{k: k\neq i}K(x^i-x^k)\right|^2.
\end{split}
\end{gather*}
\end{proof}

The following simple fact of probability will be useful later.
\begin{lemma}\label{lmm:normofrandomsum}
Fix $i\in\{1,\ldots, N\}$. Let $\mathcal{C}_{\theta}$ be the random batch of size $p$ that contains $i$ in the random division. 
Let $Y_i$ ($1\le i\le N$) be $N$ random variables (or random vectors) that are independent of $\mathcal{C}_{\theta}$. Then,
\begin{gather}
\left\|\frac{1}{p-1}\sum_{j\in\mathcal{C}_{\theta},j\neq i}Y_j\right\|\le \max_j \|Y_j\|.
\end{gather}
\end{lemma}
\begin{proof}
We again use $I(i,j)=1$ to indicate that $i,j$ are in the same batch.  By the definition and independence,
\[
\begin{split}
\left\|\frac{1}{p-1}\sum_{j\in\mathcal{C}_{\theta},j\neq i}Y_j\right\|^2
&=\frac{1}{(p-1)^2}\mathbb{E}\left|\sum_{j: j\neq i}I(i,j)Y_j\right|^2 \\
&=\frac{1}{(p-1)^2}\sum_{j,k:j\neq i,k\neq i}\mathbb{E}(I(i,j)I(i,k))\mathbb{E}(Y_j\cdot Y_k)\\
&\le (\max_j \|Y_j\|)^2 \frac{1}{(p-1)^2}\sum_{j,k:j\neq i,k\neq i}\mathbb{E}(I(i,j)I(i,k))\\
&=(\max_j \|Y_j\|)^2 \left\|\frac{1}{p-1}\sum_{j\in\mathcal{C}_{\theta},j\neq i} 1 \right\|^2.
\end{split}
\]
The claim thus follows. Note that the independence is used in the second equality.
\end{proof}

We now move onto the estimates of $X$ and $\tilde{X}$.
\begin{lemma}\label{lmm:momcond}
Suppose the coupling constructed in \eqref{eq:coupling} and Assumption \ref{ass:basicassumption} hold. Then, for any $q\ge 2$, there exists a constant $C_q$ independent of $N$ such that for any $i$
\begin{gather}\label{eq:momentcontrol}
\sup_{t\ge 0}(\mathbb{E}|X^i(t)|^q+\mathbb{E}|\tilde{X}^i(t)|^q)\le C_q.
\end{gather}
Besides, for any $m>0$ and $q\ge 2$,
\begin{gather}\label{eq:conditionmoment}
\sup_{t\in [t_{m-1}, t_m)}\left|\mathbb{E}(|\tilde{X}^i(t)|^q |\mathcal{F}_{m-1})\right|
\le C_1|\tilde{X}^i(t_{m-1})|^q+C_2.
\end{gather}
holds almost surely.

Moreover, for $t\in [t_{m-1}, t_m)$ and all $i$, there exist a constant $C>0$ independent of $\xi_m$, $m$, $ N$ and an index $q>0$ such almost surely that 
\begin{gather}\label{eq:basicconditon}
|\mathbb{E}(\tilde{X}^i(t)-\tilde{X}^i(t_{m-1}) | \mathcal{F}_{m-1})|\le C (1+|\tilde{X}^i(t_{m-1})|^{q})\tau,
\end{gather}
and in $L^2(\Omega)$ that
\begin{gather}\label{eq:condition1}
\Big\|\mathbb{E}\left(|\tilde{X}^i(t)-\tilde{X}^i(t_{m-1})|^2 | \mathcal{F}_{m-1}\right) \Big\| \le C \tau.
\end{gather}
\end{lemma}

\begin{proof}
Consider system \eqref{eq:interactingps} first. By It\^o's calculus,
\begin{gather*}
\begin{split}
\frac{d}{dt}\mathbb{E}|X^i|^q=& ~q\mathbb{E}|X^i|^{q-2}\left(-X^i\cdot\nabla V(X^i)+\frac{1}{N-1}\sum_{j: j\neq i}X^i\cdot K(X^i-X^j)\right) \\
&+\frac{1}{2}q(q+d-2)\sigma^2\mathbb{E}|X^i|^{q-2}.
\end{split}
\end{gather*}
Note that
\[
X^i\cdot\nabla V(X^i)=(X^i-0)\cdot (\nabla V(X^i)-\nabla V(0))+X^i\cdot\nabla V(0) \ge r|X^i|^2+X^i\cdot\nabla V(0).
\]
Recalling also that $K$ is bounded,
\begin{multline}\label{eq:lmmaux1}
\frac{d}{dt}\mathbb{E}|X^i|^q\le -qr \mathbb{E}|X^i|^q
+q(\|K\|_{\infty}+|\nabla V(0)|)\mathbb{E}|X^i|^{q-1}
+\frac{1}{2}q(q+d-2)\sigma^2\mathbb{E}|X^i|^{q-2}.
\end{multline}
By Young's inequality,
\[
\mathbb{E}(|X^i|^{q-1})
\le \frac{(q-1)\nu}{q} \mathbb{E}(|X^i|^{q})
+\frac{1}{q\nu^{q-1}},
\]
the second term on the right hand side of \eqref{eq:lmmaux1} is therefore controlled. 
If $q=2$, the last term on the right hand side of \eqref{eq:lmmaux1} is controlled by a constant. Otherwise, one can apply Young's inequality similarly to control it with $\mathbb{E}(|X^i|^{q})$. Clearly, when choosing $\nu$ fixed but small enough, $\mathbb{E}|X_i|^q$ can be uniformly bounded in time for any $q\ge 2$, and the bound is independent of $N$.

For $\tilde{X}^i$, we first consider a given random division so that the equation is given by
\begin{gather}\label{eq:tildeXexperiment}
d\tilde{X}^i=-\nabla V(\tilde{X}^i)dt+\frac{1}{p-1}\sum_{j\in \mathcal{C}_{\theta},j\neq i}K(\tilde{X}^i-\tilde{X}^{j})\,dt+\sigma dB^i,
\end{gather}
where $\mathcal{C}_{\theta}$ is the random batch that contains $i$ from the random division at $t_{m-1}$, or $\xi_{m-1}$.
Now, consider that $t\in [t_{m-1}, t_m)$. Conditioning on $\mathcal{F}_{m-1}$ and applying It\^o's calculus on $[t_{m-1}, t_m)$, one also has
\begin{multline*}
\frac{d}{dt}\mathbb{E}(|\tilde{X}^i|^q|\mathcal{F}_{m-1})=q\mathbb{E}\Big[|\tilde{X}^i|^{q-2}(-\tilde{X}^i\cdot\nabla V(\tilde{X}^i)+\\
\tilde{X}^i\cdot \frac{1}{p-1}\sum_{j\in \mathcal{C}_{\theta},j\neq i}K(\tilde{X}^i-\tilde{X}^{j})))| \mathcal{F}_{m-1}\Big] 
+\frac{1}{2}q(q+d-2)\sigma^2\mathbb{E}(|\tilde{X}^i|^{q-2}|\mathcal{F}_{m-1}).
\end{multline*}
Using similar estimates,
\[
\frac{d}{dt}\mathbb{E}(|\tilde{X}^i|^q|\mathcal{F}_{m-1})\le -r_1  \mathbb{E}(|\tilde{X}^i|^q|\mathcal{F}_{m-1})+C_1,
\]
for some $r_1>0$ and  constant $C_1$ that are deterministic. This clearly yields \eqref{eq:conditionmoment}.
Taking expectation about the randomness in $\mathcal{F}_{m-1}$ on both sides, one then obtains the same inequality as for $X$. Hence, \eqref{eq:momentcontrol} follows.

Equation \eqref{eq:tildeXexperiment} gives
\begin{gather*}
\begin{split}
\mathbb{E}\Big(\tilde{X}^i(t)-\tilde{X}^i(t_{m-1}) \Big| \mathcal{F}_{m-1}\Big)
=& -\int_{t_{m-1}}^t\mathbb{E}(\nabla V(\tilde{X}^i)| \mathcal{F}_{m-1})\,ds \\
 &+\int_{t_{m-1}}^t\mathbb{E}\left( \frac{1}{p-1}\sum_{j\in \mathcal{C}_{\theta},j\neq i}K(\tilde{X}^i-\tilde{X}^{j}) | \mathcal{F}_{m-1}\right)\,ds.
\end{split}
\end{gather*}
Note that $K$ is bounded and $|\nabla V(x)|\le C(1+|x|^q)$ for some $q>0$.   
Together with the first part of this lemma (moment control), this implies \eqref{eq:basicconditon}.

For the last claim, It\^o's formula implies that
\begin{multline*}
\frac{d}{dt}\mathbb{E}\left[(\tilde{X}^i(t)-\tilde{X}^i(t_{m-1}))^2|\mathcal{F}_{m-1} \right] \\
=2\mathbb{E}\left[ \Big(\tilde{X}^i(t)-\tilde{X}^i(t_{m-1})\Big)\cdot(-\nabla V(\tilde{X}^i)+\frac{1}{p-1}\sum_{j\in \mathcal{C}_{\theta},j\neq i}K(\tilde{X}^i-\tilde{X}^{j}))
|\mathcal{F}_{m-1} \right]
+\sigma^2.
\end{multline*}
Similarly,
\[
\left\|\mathbb{E}\left[ \Big(\tilde{X}^i(t)-\tilde{X}^i(t_{m-1})\Big)\cdot\left(-\nabla V(\tilde{X}^i)+\frac{1}{p-1}\sum_{j\in \mathcal{C}_{\theta},j\neq i}K(\tilde{X}^i-\tilde{X}^{j})\right)
|\mathcal{F}_{m-1} \right] \right\|\le C
\]
is uniformly bounded by H\"older's inequality for conditional expectations (though $\tilde{X}^i(t)-\tilde{X}^i(t_{m-1})$ is small, there is no need to get finer bound since we have extra $\sigma^2$). The last claim also follows.

\end{proof}

The following gives the control on $Z$.
\begin{lemma}\label{lmm:Zest}
For $t\in [t_{m-1}, t_m)$, 
\begin{gather}
\begin{split}
 & \|Z^i(t)-Z^i(t_{m-1})\|\le C \tau,\\
 & \Big|\mathbb{E}\left((Z^i(t)-Z^i(t_{m-1}))\cdot\chi_{m,i}(X(t))\right)\Big|\le \|\Lambda_i\|_{\infty}\frac{\tau}{p-1} 
+C\left(\tau^2+\|Z^i(t)\|\tau\right).
\end{split}
\end{gather}
\end{lemma}
\begin{proof}
By the coupling in \eqref{eq:coupling}, $Z^i$ satisfies on $t\in [t_{m-1}, t_m)$
\begin{multline}\label{eq:eqnforZi}
dZ^i=-(\nabla V(\tilde{X}^i)-\nabla V(X^i))dt
+\frac{1}{p-1}\sum_{j\in \mathcal{C}_{\theta},j\neq i}K(\tilde{X}^i-\tilde{X}^{j})\,dt-\frac{1}{N-1}\sum_{j: j\neq i}K(X^i-X^j)\,dt.
\end{multline}
Since $\nabla V$ has polynomial growth, the claim for $\|Z^i(t)-Z^i(t_{m-1})\|$ is then an easy consequence of the $q$-moment estimates in Lemma \ref{lmm:momcond}.

Moreover, 
\[
|\nabla V(\tilde{X}^i)-\nabla V(X^i)|\le \int_0^1|(\tilde{X}^i-X^i)\cdot \nabla^2V((1-z)\tilde{X}^i+zX^i)|\,dz,
\]
and $\nabla^2V((1-z)\tilde{X}^i+zX^i)$ is controlled by $C(|\tilde{X}^i|^{q_1}+|X^i|^{q_1})$ for some $q_1>0$. Hence,
\begin{gather*}
\begin{split}
\mathbb{E}|\nabla V(\tilde{X}^i)-\nabla V(X^i)||\chi_{m,i}(X(t'))|
& \le \|\chi_{m,i}(X(t'))\|_{\infty}\|\tilde{X}^i-X^i\| (\mathbb{E}(|\tilde{X}^i|^{q_1}+|X^i|^{q_1})^2)^{1/2} \\
&\le C\|Z^i(t)\|,
\end{split}
\end{gather*}
by Lemma \ref{lmm:momcond}, where $t'\in [t_{m-1}, t_m]$ is arbitraray.

Integrating \eqref{eq:eqnforZi} in time over $[t_{m-1}, t]$, then dotting with $\chi_{m,i(X(t))}$, and taking the expectation,  one gets  
\begin{multline*}
\left| \mathbb{E}((Z^i(t)-Z^i(t_{m-1}))\cdot\chi_{m,i}(X(t))) \right|
\le C\int_{t_{m-1}}^t\|Z^i(s)\| ds+\\
\int_{t_{m-1}}^{t}\mathbb{E}\left[\left(\frac{1}{p-1}\sum_{j\in\mathcal{C}_{\theta},j\neq i}\delta K_{ij}(s)\right)\cdot\chi_{m,i}(X(t))\right]\,ds
+\mathbb{E}\int_{t_{m-1}}^{t}\chi_{m,i}(X(s))\cdot\chi_{m,i}(X(t))\,ds,
\end{multline*}
where
\begin{gather}\label{eq:deltaKnewdef}
\delta K_{ij}(s):=K(\tilde{X}^i(s)-\tilde{X}^j(s))-K(X^i(s)-X^j(s)).
\end{gather}

Since $\|Z^i(s)\|\le \|Z^1(t)\|+C\tau$, the first term is controlled by $C\|Z^i(t)\|\tau+C\tau^2$. Since $K$ is Lipschitz continuous,
\[
|\delta K_{ij}(s)|\le L(|Z^i(s)|+|Z^j(s)|).
\]
Hence,
\begin{gather}
\left\|\left(\frac{1}{p-1}\sum_{j\in\mathcal{C}_{\theta},j\neq i}\delta K_{ij}(s)\right)\right\|
\le L\left(\|Z^i(s)\|+\left\|  \frac{1}{p-1}\sum_{j\in\mathcal{C}_{\theta},j\neq i}|Z^j(s)|\right\|\right).
\end{gather}
Now, $Z^j(s)$ depends on $\mathcal{C}_{\theta}$ so it cannot be estimated directly. 
Dotting \eqref{eq:eqnforZi} with $Z^i$, one has $
\frac{1}{2}\frac{d}{dt}|Z^i|^2\le C|Z^i|$ almost surely,
which gives 
\[
|Z^i(s)|\le |Z^i(t_{m-1})|+C\tau
\]
 almost surely. Since $Z^i(t_{m-1})$ is independent of $\mathcal{C}_{\theta}$, Lemma \ref{lmm:normofrandomsum} then gives us that
 \[
\left\|  \frac{1}{p-1}\sum_{j\in\mathcal{C}_{\theta},j\neq i}|Z^j(s)|\right\| \le \|Z^1(t_{m-1})\|+C\tau
 \le \|Z^1(t)\|+\|Z^1(t)-Z^1(t_{m-1})\|+C\tau.
 \]
 Since $\|Z^1(t)-Z^1(t_{m-1})\|\le C\tau$ by what has been just proved, this term is bounded by $\|Z^1(t)\|+C\tau$.
 Hence,
 \begin{gather}\label{eq:randomsumofdeltaK}
\left\|\left(\frac{1}{p-1}\sum_{j\in\mathcal{C}_{\theta},j\neq i}\delta K_{ij}(s)\right)\right\|
\le 2L(\|Z^1(t)\|+C\tau).
\end{gather}

Since $X$ is independent of $\mathcal{C}_{\theta}$, applying Lemma \ref{lmm:averagefunc},
\[
\mathbb{E}\Big(\chi_{m,i}(X(s))\cdot\chi_{m,i}(X(t))\Big)
\le \frac{1}{p-1}\|\Lambda_i\|_{\infty},
\] 
and the claim follows.
\end{proof}

Now, we are ready to prove Theorem \ref{thm:mainresult}.
\begin{proof}[Proof of Theorem \ref{thm:mainresult}]
With the coupling \eqref{eq:coupling}, the continuous process $Z^i$ satisfies for $t\in [t_{m-1}, t_m)$ that
\begin{gather*}
\begin{split}
dZ^i =& -(\nabla V(\tilde{X}^i)-\nabla V(X^i))dt \\
          &+\frac{1}{N-1}\sum_{j: j\neq i}(K(\tilde{X}^i-\tilde{X}^j)-K(X^i-X^j)) dt +\chi_{m,i}(\tilde{X}(t))dt.
\end{split}
\end{gather*}
Using the strong convexity of $V$ and Lipscthitz continuity of $K$, 
\[
\frac{1}{2}d\mathbb{E}\sum_{i=1}^N(Z^i)^2
\le -(r-2L)\mathbb{E}\sum_{i=1}^N (Z^i)^2\,dt
+\sum_{i=1}^N \mathbb{E}Z^i(t)\cdot \chi_{m,i}(\tilde{X}(t))\,dt,
\]
where $\tilde{X}=(\tilde{X}^1, \ldots, \tilde{X}^N)\in \mathbb{R}^{Nd}$.  Due to the exchangeability, 
\begin{gather}
R(t):=\frac{1}{N}\sum_{i=1}^N \mathbb{E}Z^i(t)\cdot \chi_{m,i}(\tilde{X}(t))=\mathbb{E}Z^1(t)\cdot \chi_{m,1}(\tilde{X}(t)).
\end{gather}
For notational convenience, we introduce
\begin{gather}
u(t):=\frac{1}{N}\mathbb{E}\sum_{i=1}^N(Z^i)^2=\|Z^1(t)\|^2.
\end{gather}
Then,
\[
\frac{d}{dt}u(t)\le -2(r-2L)u(t)+2R(t).
\]

We now estimate $R(t)$. We rewrite it as
\begin{gather}\label{eq:Rapprox}
\begin{split}
R(t)= & \mathbb{E}Z^1(t_{m-1})\cdot \chi_{m,1}(\tilde{X}(t_{m-1}))
+\mathbb{E}Z^1(t_{m-1})\cdot (\chi_{m,1}(\tilde{X}(t))-\chi_{m,1}(\tilde{X}(t_{m-1})))\\
         &   + \mathbb{E}(Z^1(t)-Z^1(t_{m-1}))\cdot \chi_{m,1}(X(t)) \\
          & +\mathbb{E}(Z^1(t)-Z^1(t_{m-1}))\cdot (\chi_{m,1}(\tilde{X}(t))-\chi_{m,1}(X(t)))\\
          =:&I_1+I_2+I_3+I_4.
\end{split}
\end{gather}

By Lemma \ref{lmm:averagefunc}, we have 
\begin{gather}\label{eq:I1constency}
\begin{split}
I_1 & = \mathbb{E}\left(\mathbb{E}(Z^1(t_{m-1})\cdot \chi_{m,1}(\tilde{X}(t_{m-1})) )|\mathcal{G}_{m-1} \right) \\
&=\mathbb{E}\left(Z^1(t_{m-1})\cdot \mathbb{E} (\chi_{m,1}(\tilde{X}(t_{m-1}))
|\mathcal{G}_{m-1}) \right)=0.
\end{split}
\end{gather}
The second equality holds because $Z^1(t_{m-1})$ is adapted to $\mathcal{G}_{m-1}$ (note that $Z^1$ is continuous in time, so $Z^1(t_{m-1}^-)=Z^1(t_{m-1})$). Note that $\mathcal{G}_{m-1}$ is independent of the random division at $t_{m-1}$ and this is why we can get $I_1=0$ by regarding $\tilde{X}(t_{m-1})$ as a given fixed point.  Equation \eqref{eq:I1constency} is the consistency that ensures convergence.

For term $I_2$, 
\begin{gather}
\begin{split}
I_2 &=\mathbb{E}\Bigg( Z^1(t_{m-1})\cdot \mathbb{E}(\chi_{m,1}(\tilde{X}(t))-\chi_{m,1}(\tilde{X}(t_{m-1})) |\mathcal{F}_{m-1}) \Bigg) \\
& \le C\|Z^1(t_{m-1})\| \|\mathbb{E}(\chi_{m,i}(\tilde{X}(s))-\chi_{m,i}(\tilde{X}(t_{m-1})) 
|\mathcal{F}_{m-1})\|.
\end{split}
\end{gather}
Clearly,
\[
\mathbb{E}(\chi_{m,i}(\tilde{X}(s))-\chi_{m,i}(\tilde{X}(t_{m-1})) |\mathcal{F}_{m-1})
=\frac{1}{p-1}\sum_{j\in \mathcal{C}_{\theta},j\neq i} \mathbb{E}(\delta \tilde{K}^{ij}|\mathcal{F}_{m-1})
-\frac{1}{N-1}\sum_{j:j\neq i}\mathbb{E}(\delta\tilde{K}^{ij}|\mathcal{F}_{m-1}),
\]
where 
\[
\delta\tilde{K}^{ij}=K(\tilde{X}^i(s)-\tilde{X}^j(s))-K(\tilde{X}^i(t_{m-1})-\tilde{X}^j(t_{m-1})).
\]
Denote $\delta\tilde{X}^j:=\tilde{X}^j(s)-\tilde{X}(t_{m-1})$. Performing Taylor expansion around $t_{m-1}$, one has
\[
\delta\tilde{K}^{ij}=\nabla K(\tilde{X}^i(t_{m-1})-\tilde{X}^j(t_{m-1}))\cdot(\delta\tilde{X}^i-\delta\tilde{X}^j)
+\frac{1}{2}M:(\delta\tilde{X}^i-\delta\tilde{X}^j)\otimes(\delta\tilde{X}^i-\delta\tilde{X}^j),
\]
with $M$ being a random variable (matrix) bounded by $\|\nabla^2K\|_{\infty}$.
By \eqref{eq:basicconditon}, one finds that
 \[
 |\mathbb{E}(\delta \tilde{K}^{ij}|\mathcal{F}_{m-1})|
 \le CL(1+|\tilde{X}^i(t_{m-1})|^q+|\tilde{X}^j(t_{m-1})|^q)\tau+C\mathbb{E}(|\delta\tilde{X}^i|^2+|\delta\tilde{X}^j|^2  |\mathcal{F}_{m-1}).
 \]
By \eqref{eq:condition1} and the fact that $\tilde{X}^i(t_{m-1})$ is independent of $\mathcal{C}_{\theta}$, Lemma \ref{lmm:normofrandomsum} then can control 
\[
I_2\le C\|Z^1(t_{m-1})\|\tau \le C\|Z^1(t)\|\tau+C\tau^2.
\]

By Lemma \ref{lmm:Zest}, the third term on the right hand side of \eqref{eq:Rapprox} is bounded by
\[
I_3  \le C\frac{\tau}{p-1}+C\|Z^1(t)\|\tau+C\tau^2.
\]

The $I_4$ term can be controlled by
\[
I_4\le \|Z^1(t)-Z^1(t_{m-1})\| \left( \Big\|\frac{1}{p-1}\sum_{j\in\mathcal{C}_{\theta}}\delta K_{ij}(t)\Big\|
+\Big\|\frac{1}{N-1}\sum_{j:j\neq i}\delta K_{ij}(t) \Big\|\right),
\]
where $\delta K_{ij}(t)$ is defined in \eqref{eq:deltaKnewdef}.
The first term inside the parenthesis on the right hand side is estimated by \eqref{eq:randomsumofdeltaK}. The second term is 
straightforward. Hence, noticing Lemma \ref{lmm:Zest}, one has 
\[
I_4\le \|Z^1(t)-Z^1(t_{m-1})\|(4L\|Z^1(t)\|+C\tau)\le C\|Z^1(t)\|\tau+C\tau^2.
\]

Hence,
\[
R(t) \le C\|Z^1(t)\|\tau+C\frac{\tau}{p-1}+C\tau^2
\]
and thus
\[
\frac{d}{dt}u(t)
\le -2(r-2L)u(t)+C\sqrt{u(t)}\tau+C\frac{\tau}{p-1}+C\tau^2.
\]
This inequality therefore implies that
\[
\sup_{t\ge 0}\mathbb{E}|\tilde{X}^1-X^1|^2=\sup_{t\ge 0}u(t) \le C\frac{\tau}{p-1}+C\tau^2.
\]
The closeness between $\mu_N^{(1)}$ and $\tilde{\mu}_N^{(1)}$ in $W_2$ distance is a simple application of definition \eqref{eq:W2}.
\end{proof}

We point out that the constant $C$ in Theorem \ref{thm:mainresult} stays bounded as $\sigma\to 0$. The proof in Theorem \ref{thm:mainresult} is valid for $\sigma=0$, when there is no Brownian motion.  In $I_1$ and $I_2$, we used
$\chi_{m,i}(\tilde{X}(t_{m-1}))$ to split the terms. If we use $\chi_{m,i}(X(t))$ to split the terms and use similar techniques
as in $I_3, I_4$,  then there is no requirement on $\nabla^2 K$ but we need $r>6L$.

The first part in Theorem \ref{thm:mainresult}  in fact says that we can approximate the trajectories of the particles, which is the convergence in strong sense.

\begin{figure}
\begin{center}
	\includegraphics[width=0.4\textwidth]{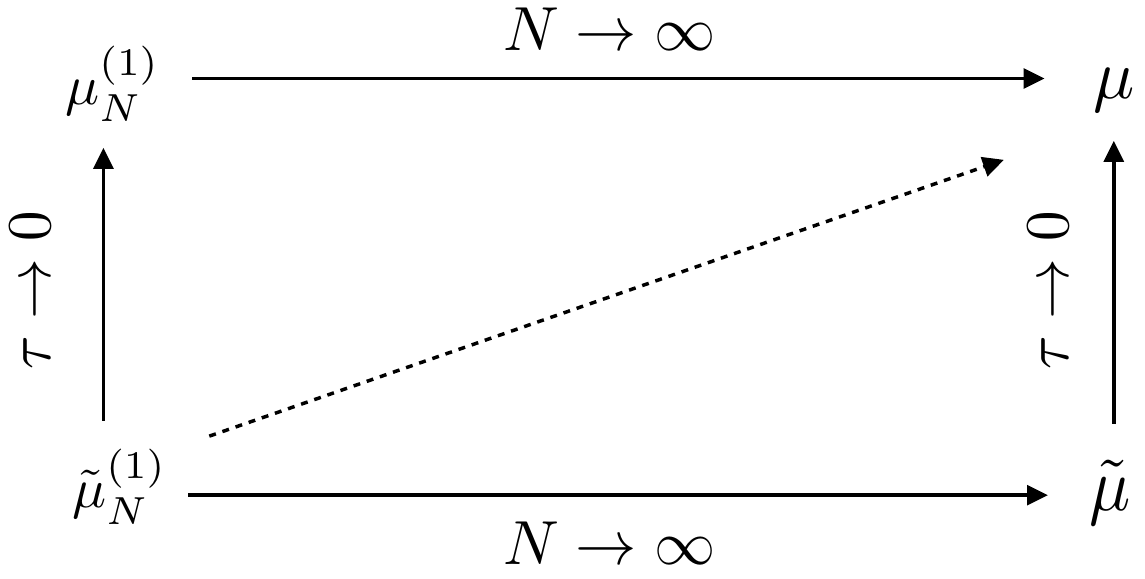}
\end{center}
\caption{Illustration for the convergence of the RBMs}
\label{fig:convergencetauN}
\end{figure}

It is straightforward to conclude that RBM-1 can approximate the mean field measure (solution to \eqref{eq:nonlinearfp}) since the mean field result is well known (see, for example, \cite{cattiaux2008,del2013uniform,durmus2018elementary}). In fact, as shown in Fig. \ref{fig:convergencetauN}, the triangle inequality yields the following.
\begin{corollary}
Suppose Assumption \ref{ass:basicassumption} holds, then
\begin{gather}
W_2(\tilde{\mu}_N^{(1)}(t), \mu(t))\le C(\sqrt{\tau}+N^{-1/2+\epsilon})
\end{gather}
for any $\epsilon>0$.
\end{corollary}
The error bound between $\tilde{\mu}_N^{(1)}$ and $\mu_N^{(1)}$ is $C\sqrt{\tau}$ uniformly in time with $C$ independent of $N$. The error bound between $\mu_N^{(1)}$ and $\mu$ is given by $CN^{-1/2+\epsilon}$ (\cite{cattiaux2008}). 
As shown in the figure, if one takes $N\to\infty$ first, one may get some measure $\tilde{\mu}(t)$, which is the mean field limit of the RBM-1 method. Analyzing this measure will be an interesting problem for the future.

As shown in Fig. \ref{fig:convergencetauN}, the RBMs are a class of Asymptotic-Preserving schemes \cite{jin2010asymptotic} for particle system \eqref{eq:interactingps} under the scalings of mean field limit, in that they approximate the particle system \eqref{eq:interactingps} with an error of $O(\tau)$ independent of $N$, and as $N \to \infty$, they become good approximations to the limiting mean-field equation \eqref{eq:nonlinearfp}.

\begin{remark}\label{rmk:nocontraction}
If the deterministic flow does not have the contraction property, we may only be able to prove the convergence of our methods on finite time interval $[0, T]$. Moreover, the mean field limit by Dobrushin's estimate holds on finite time interval \cite{golse2003mean,jabinquantitative,benachour1998}. Consequently, it is possible to show that 
\begin{gather}
W_2(\tilde{\mu}^{(1)}_N(t), \mu(t))\le C(T)(\sqrt{\tau}+N^{-1/2}),~\forall t\in [0, T],
\end{gather}
for general interaction potentials, where the coefficient depends on $T$ now.
\end{remark}

\begin{remark}
For $N\gg 1$ when the theory of mean field limit takes effect, the chaos could be created or propagated (in terms of the distribution of exchangeable particles). In other words, the $j$-marginal will be the tensor product of $j$ copies of one particle distribution. In this case, the joint distribution could be of low rank, which helps RBMs: the number of time steps $N_T$ can be much smaller than $N/p$ to get reasonably good approximation ability. This is definitely something that can be explored in the future.
\end{remark}

\begin{remark}
The fast Monte Carlo method in \cite{drineas2006fast1,drineas2006fast2} uses subsampling in columns or rows for matrix-matrix multiplication. The idea of subsampling as unbiased estimator is the same as the one for RBMs. 
However, besides subsampling, another key idea of RBM is re-subsampling at later time intervals so that the dynamics will be averaged out in time, which differs from the random algorithms in matrix computations. This Law of Large Numbers type feature used in time guarantees the convergence of RBMs (as indicated by the Law of Large Number type error bound in Theorem \ref{thm:mainresult}). 
\end{remark}

\begin{remark}\label{rmk:extensiontoq}
The analysis here should also work for \eqref{eq:interactingps0}, provided some analogies of Lemma \ref{lmm:averagefunc} hold (as discussed in section \ref{sec:diss}). For techanical needs, one will assume $r$ large enough so that the semigroup given by the deterministic flow is again a contraction. Besides, the analysis can be carried over directly to the $O(1)$ time simulation of second order systems such as Hamiltonian dynamics, and we provide such an example in Appendix \ref{sec:Hamil}.
\end{remark}

\begin{remark}
If one considers particles associated with charges or masses, they are not exchangeable and the current error analysis does not work directly. The consistency lemma (Lemma \ref{lmm:averagefunc}) in fact has nothing to do with exchangeability and can be generalized to interactions that dependent on the specific chosen particles.  The current proof of convergence seems to rely on exchangeability, but one can consider the charge or mass as an extended coordinate \cite{liu2017random}. With this new insight, the new 'particles' become exchangeable. We believe RBMs also work for such systems but we feel it better to leave this as a subsequent project. 
\end{remark}

\section{Numerical verification}\label{sec:numverify}

In this section, we run some numerical tests to evaluate the RBMs and verify our theory in section \ref{sec:error}. 
The first example is a simple artificial example to test the dependence of the errors on $N$ and $\tau$, which satisfies the conditions in Theorem \ref{thm:mainresult}. The second example is the Dyson Brownian motion, and it does not satisfy the conditions in Theorem \ref{thm:mainresult} as the kernel is singular. RBMs work well for both examples, which implies that the algorithms can be effective for systems far beyond the systems mentioned in section \ref{sec:error}.

\subsection{A simple test example}\label{subsec:1dsimpletest}

\begin{figure}
\begin{center}
	\includegraphics[width=0.8\textwidth]{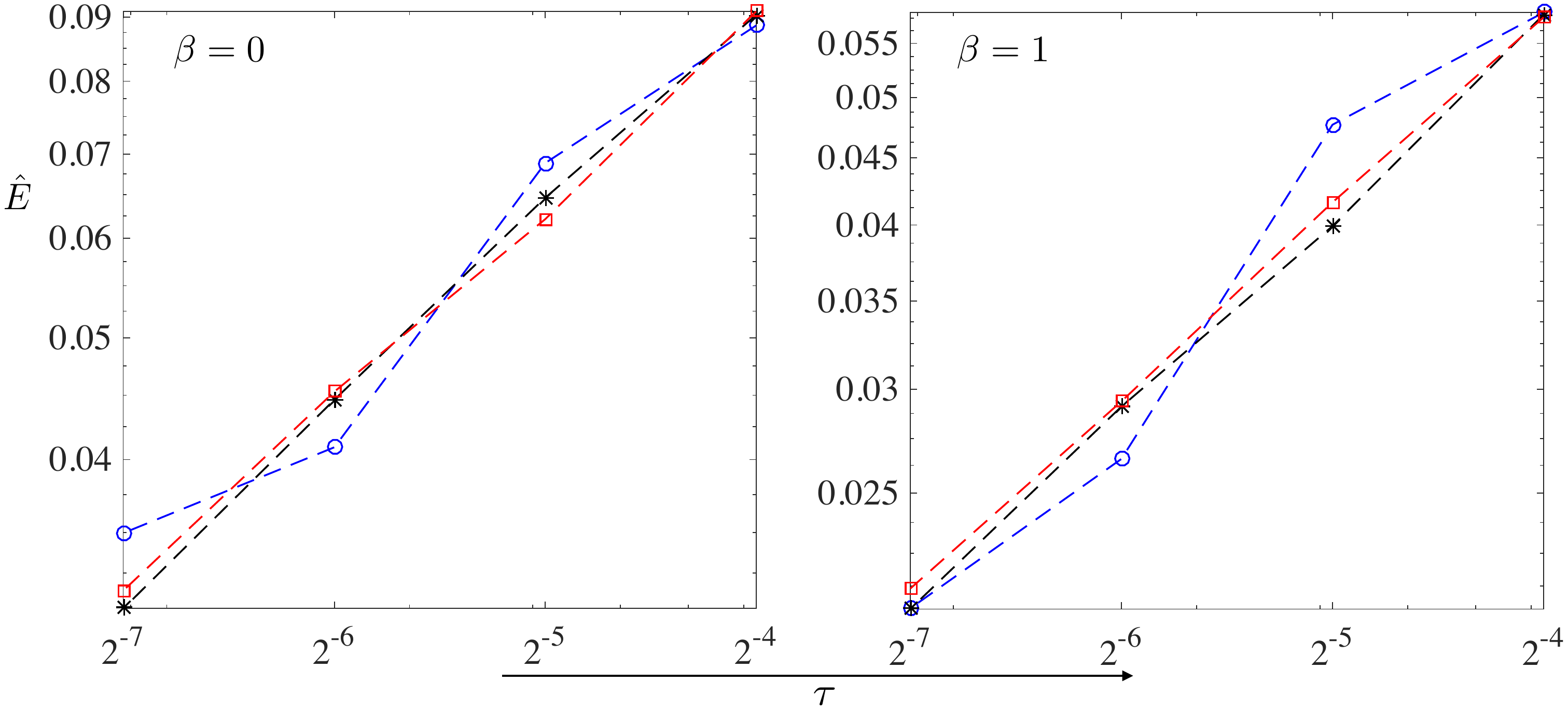}
\end{center}
\caption{Error of RBM-1 versus time step. Blue circle is for $N=50$, black star is for $N=500$, red square is for $N=2000$.}
\label{fig:errdepN}
\end{figure}

We now use the following simple test example to check how the error in RBM-1 depends on $N$, $\tau$ and $T$ ($T$ is the time point where we compute the numerical solutions). Here, the spatial dimension is $1$ ($d=1$)
\begin{gather}
\dot{X}^i=-\beta X^i+\frac{1}{N-1}\sum_{j: j\neq i}\frac{X^i-X^j}{1+|X^i-X^j|^2}.
\end{gather}
The interaction is clearly smooth, bounded and with bounded derivatives. Moreover, it has a long-range interaction. 

In principle, to evaluate $E(T)=\sqrt{\mathbb{E}|\tilde{X}^1(T)-X^1(T)|^2}$, we need to run many independent experiments and use empirical mean for the approximation. Doing this is clearly very expensive. Alternatively, we only run one experiment and use
\begin{gather}\label{eq:errorformula}
\hat{E}(T):=\sqrt{\frac{1}{N}\sum_{i=1}^N|\tilde{X}^i(T)-X^i(T)|^2}
\end{gather}
to approximate $E(T)$.

In Fig. \ref{fig:errdepN}, we show the numerical results for $T=1$. The initial distribution is taken from
\[
\rho_0(x)=\frac{\sqrt{4-x^2}}{2\pi},
\]
by the Metropolis-Hastings MCMC algorithm \cite{hastings1970monte}.
The reference solution $X^i(T)$ is obtained by solving the fully coupled system using the forward Euler scheme with $\tau=2^{-15}$. The solution $\tilde{X}^i(T)$ is generated by RBM-1 with $p=2$. Each step is solved by the forward Euler method with $\tau$ from $2^{-7}$ to $2^{-4}$. We considered $N=50, 500, 2000$ respectively.  We plot the error $\hat{E}(T)$ versus $\tau$ for these three $N$ values. The first picture in Fig. \ref{fig:errdepN} is for $\beta=0$ while the second picture in Fig. \ref{fig:errdepN} is for $\beta=1$. Clearly, the error is insensitive to the change of $N$. When $N$ is small, like $N=50$, the fluctuation in the error is kind of clear. When $N$ is large, in the log-log scale, the curve is already close to straight lines with slope approximately $0.5$, meaning that the error indeed decays like $\sqrt{\tau}$.

\begin{figure}
\begin{center}
	\includegraphics[width=0.8\textwidth]{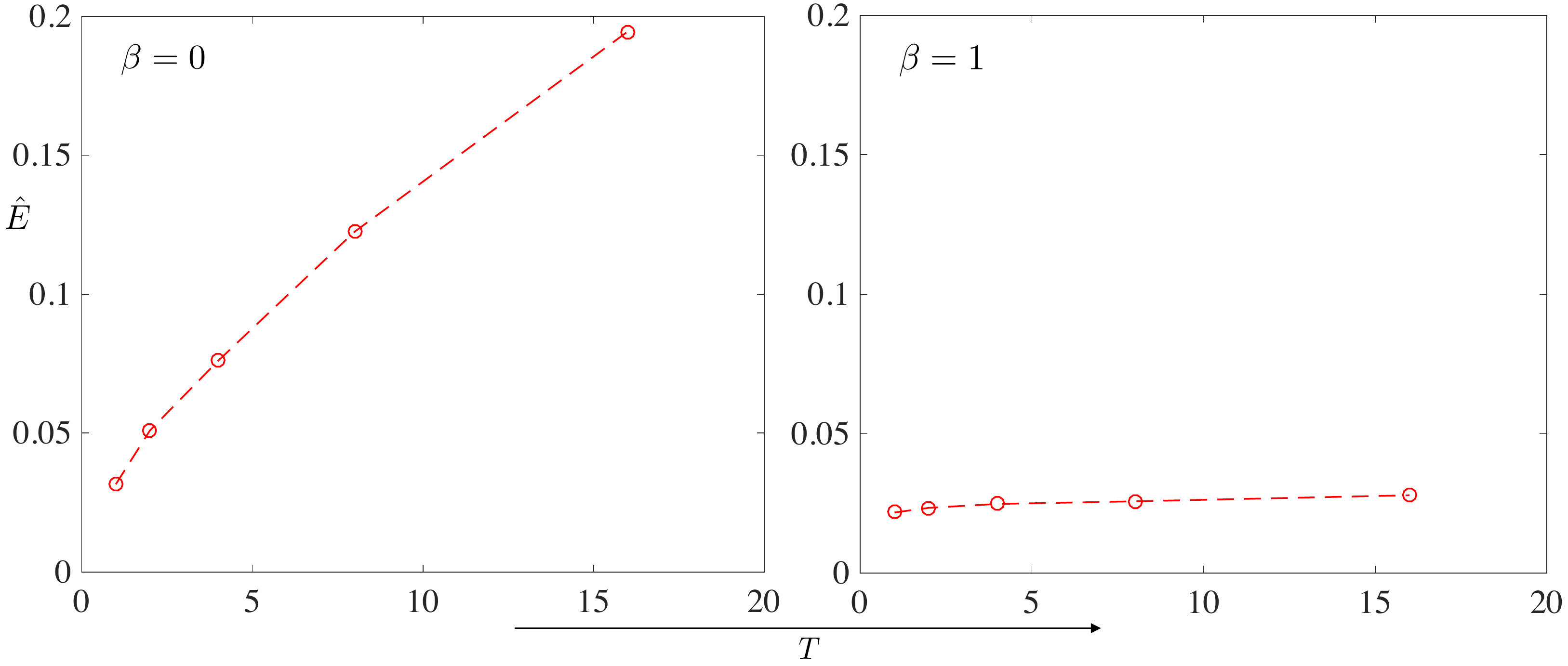}
\end{center}
\caption{Error for RBM-1 versus time.}
\label{fig:errdepT}
\end{figure}

In Fig. \ref{fig:errdepT}, we take $N=500$. The reference solution is again computed by solving the fully coupled system using the forward Euler scheme with $\tau=2^{-15}$. The algorithm is performed by taking $\tau=2^{-7}$. If there is confining potential, the error stays bounded as $T$ increases. However, if there is no confining potential, the error clearly grows, consistent with Remark \ref{rmk:nocontraction}. This is indeed natural even for usual ODE discretization for the fully coupled system \eqref{eq:interactingps}. In fact, if there is no confining potential, the numerical error grows with $T$ for the forward Euler method.

\begin{figure}
\begin{center}
	\includegraphics[width=0.6\textwidth]{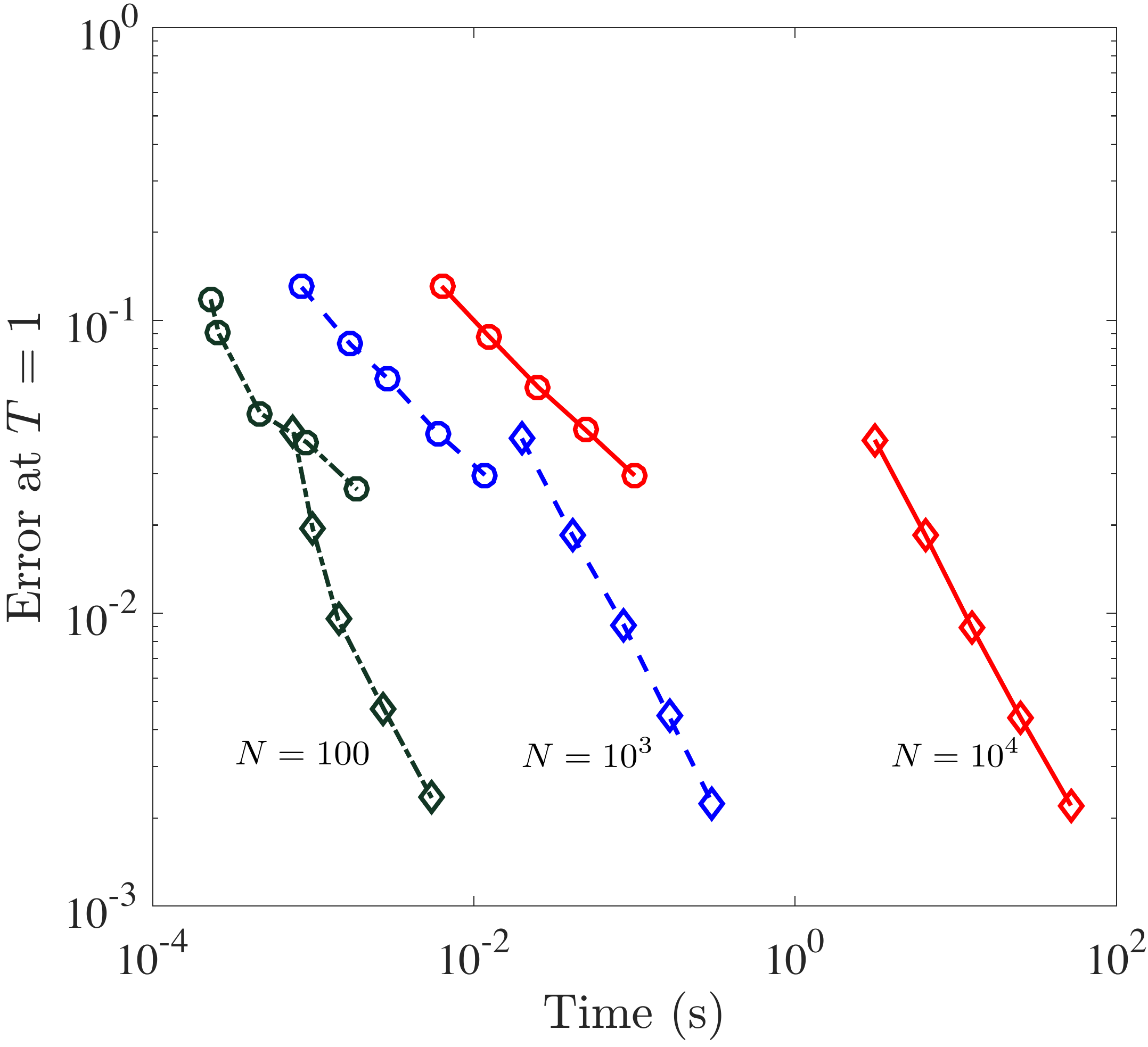}
\end{center}
\caption{Error versus CPU time. Lines with circles are for RBM-1 while lines with diamonds
are for direct computation using forward Euler discretization.}
\label{fig:cpu}
\end{figure}

To test the practical performance of RBM-1, we show the error versus true CPU time in Fig. \ref{fig:cpu} for $N=100, 10^3, 10^4$ respectively. The experiment was done using MATLAB R2015b on macOS Mojave system with 1.6 GHz Intel Core i5 processor and 4GB DDR3 memory. The data were collected by computing the times and errors using forward Euler again for steps $\tau=2^{-6}-2^{-1}$, and then plotting errors versus times.  The lines with circles are for RBM-1 while lines with diamonds are for direct computation using forward Euler discretization. One can see clearly that
for the error tolerance $\epsilon\gtrsim 10^{-2}$ when random algorithms are acceptable, RBM-1 saves a lot of time
compared with the direct ODE solvers. For example, when $N=10^3$, 
RBM-1 is faster to achieve the given error tolerance $\epsilon\ge 10^{-2}$; for $N=10^4$, RBM is much faster in the regime shown in the figure.
However, since RBMs converge with only $1/2$ order, the time needed for RBM can be larger if one desires very accurate result and $N$ is not big, and even for this case, one can consider some variance reduction method to improve the accuracy of RBMs. For $N\gtrsim 10^6$, direct method takes very long time (it scales like $N^2$), and uses much memory resource, so the direct method is already unacceptable. RBMs can be run in a reasonable time amount, and this is exactly one of the advantages of RBMs.

We now modify the above example to a second order system as studied in Appendix \ref{sec:Hamil}, which is a Hamiltonian system, to show that RBM-1 also works for second order systems on finite time. How to develop random algorithms for Hamiltonian systems for long time is an interesting question, and we think it better to leave it for future research. In particular, we consider the following system of equations for $i=1,\ldots, N$:
\begin{gather}
\left\{
\begin{split}
&\dot{X}^i=V^i,\\
&\dot{V}^i=\frac{1}{N-1}\sum_{j: j\neq i}\frac{X^i-X^j}{1+|X^i-X^j|^2}.
\end{split}
\right.
\end{gather}

For time integration, we use the Verlet scheme
\begin{gather}
\left\{
\begin{split}
&X_1^i=X_0^i+V_0^i \tau+\frac{1}{2}F_0^i\tau^2,\\
&X_{n+1}^i=2*X_n^i-X_{n-1}^i+F_n^i\tau^2,~~n\ge 1.
\end{split}
\right.
\end{gather}
For the reference solution, $F^i$ is the full force using all the particles, while for the random algorithms, $F^i$ is computed using RBM-1 with $p=2$ as before. We sample the initial positions again from $\rho_0(x)=\frac{\sqrt{4-x^2}}{2\pi}$ while sample the initial velocities  from the normal distribution  $\mathcal{N}(0, 1)$.
The error as in \eqref{eq:errorformula} is again computed at $T=1$. We consider $N=50, 500, 2000$ respectively and the reference solution is obtained using $\tau=2^{-15}$.  The results are shown in Fig. \ref{fig:second}. The curves are close to straight lines in loglog scale with slope approximately $0.5$, meaning that the errors indeed decay like $\sqrt{\tau}$ for second order systems as well.

\begin{figure}
\begin{center}
	\includegraphics[width=0.6\textwidth]{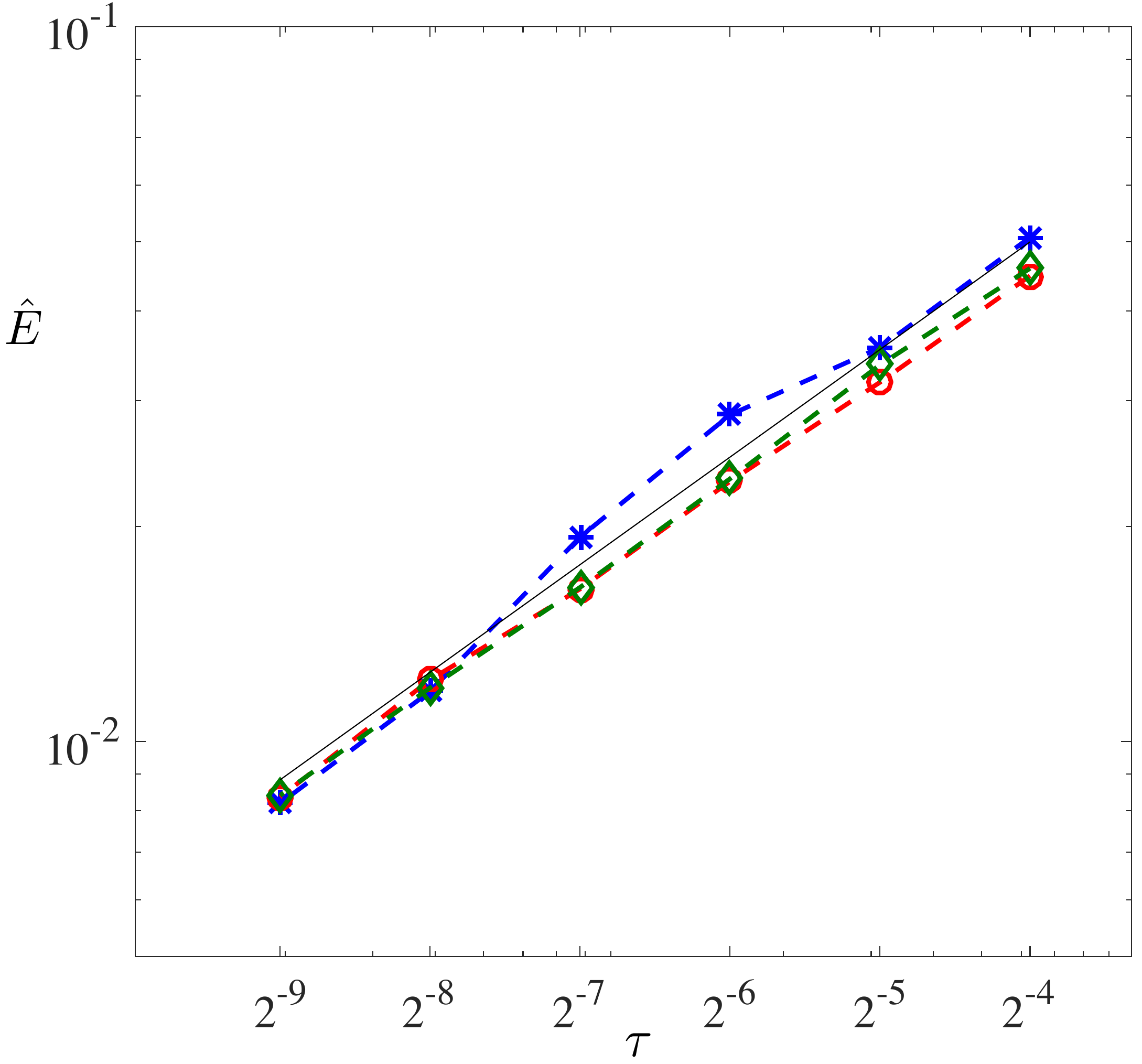}
\end{center}
\caption{$O(1)$ time simulation of a 1D Hamiltonian system. Blue star curve is for $N=50$, red circle curve is for $N=500$ while green diamond curve is for
$N=2000$. The black curve is $E=0.2 \tau^2$ for reference.}
\label{fig:second}
\end{figure}

\subsection{The Dyson Brownian motion}\label{sec:dyson}
Now, we consider a typical example in random matrix theory \cite{tao2012,erdos2017} to test the difference between RBM-1 and RBM-r. 
The random matrix we consider is a Hermitian matrix valued Ornstein-Uhlenbeck process 
\begin{gather}
dA=-\beta A\,dt+\frac{1}{\sqrt{N}}dB,
\end{gather}
where the matrix $B$ is a Hermitian matrix consisting of some Brownian motions. In particular, the diagonal elements are independent standard Brownian motions. The off-diagonal elements in the upper triangular half are of the form $\frac{1}{\sqrt{2}}(B_R+i B_I)$ where $B_R$ and $B_I$ are independent standard Brownian motions. The lower triangular half elements are determined using the Hermitian property. By It\^o's calculus \cite{dyson1962,tao2012,erdos2017}, it can be shown that the eigenvalues of $A$ satisfy the following  system of SDEs ($1\le j\le N$), called the Dyson Brownian motion:
\begin{gather}
d\lambda_j(t) =-\beta \lambda_j(t)\,dt+\frac{1}{N}\sum_{k: k\neq j}\frac{1}{\lambda_j-\lambda_k}dt
+\frac{1}{\sqrt{N}} dB_j,
\end{gather}
where $\{B_j\}$'s are independent standard Brownian motions. The Brownian motion effect is small when $N$ is large. This system therefore should have similar effects as system \eqref{eq:interactingps} with $\sigma=0$. The limiting equation for $N\to\infty$ is given by \cite{ggl19}
\begin{gather}\label{eq:dysonlimiteq}
\partial_t\rho(x,t)+\partial_x(\rho(u-\beta x))=0, ~~u(x, t)=\pi(H\rho)(x, t),
\end{gather}
where $\rho$ is the density for $\lambda$ as $N\to\infty$, $H(\cdot)$ is the Hilbert transform on $\mathbb{R}$, and $\pi=3.14\ldots$ is the circumference ratio.

Below we consider
\begin{gather}
\beta=1.
\end{gather}
It is shown that the corresponding limiting equation \eqref{eq:dysonlimiteq} has an invariant measure, given by the semicircle law:
\begin{gather}\label{eq:semicircle}
\rho(x)=\frac{1}{\pi}\sqrt{2-x^2}
\end{gather}

\begin{figure}
\begin{center}
	\includegraphics[width=1\textwidth]{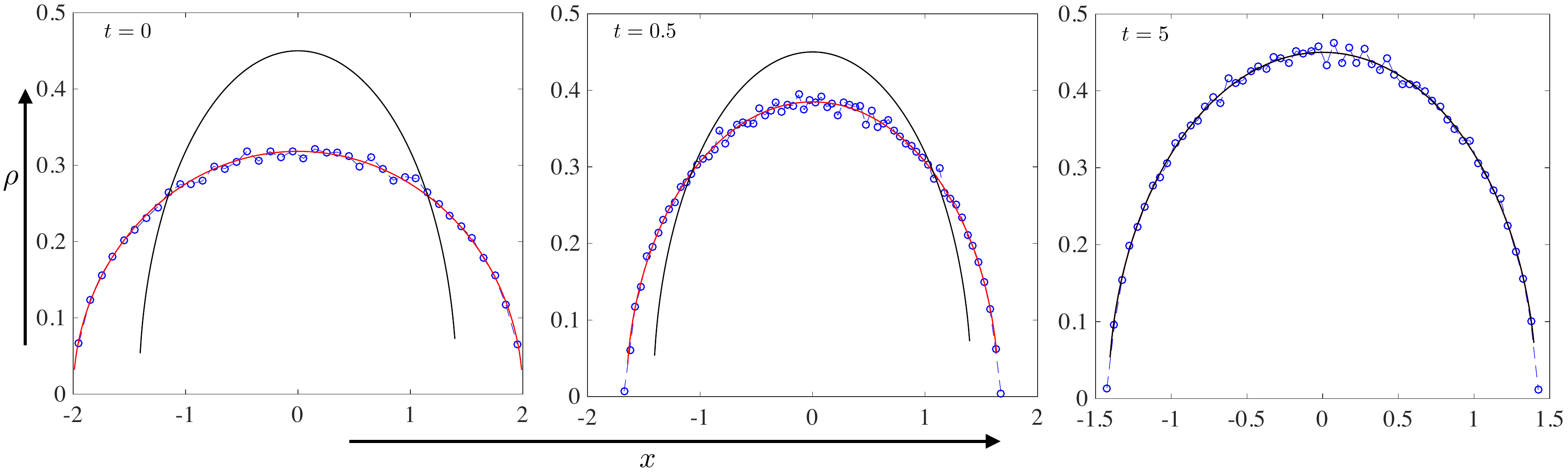}
\end{center}
\caption{The RBM-1 simulation of the Dyson Brownian motion. The empirical densities at various times are plotted. The red curve is the density distribution predicted by the analytic solution \eqref{eq:analyticaldis}. The black curve is the equilibrium semicircle law \eqref{eq:semicircle}.}
\label{fig:dysonbmDivision}
\end{figure}

To numerically test the behavior of our methods, we note an analytic solution to the limiting equation \eqref{eq:dysonlimiteq}
\begin{gather}\label{eq:analyticaldis}
\rho(x, t)=\frac{\sqrt{2\sigma(t)-x^2}}{\sigma(t)\pi},~~\sigma(t)=1+e^{-2t}.
\end{gather}
The prefactor for the interaction in the equation is $1/N$. For convenience, we change it to $1/(N-1)$ without introducing significant difference. For each iteration, since the force is singular, we adopted the splitting strategy mentioned in section \ref{sec:diss}. 
In particular, define
\begin{gather}
X^{ij}:=X^i-X^j.
\end{gather}
The SDE solving step in RBM-1 is given by
\begin{itemize}
\item  \[
 \begin{split}
        Y_m^i=\frac{1}{2}(X_{m-1}^i+X_{m-1}^j)
                      +\mathrm{sgn}(X_{m-1}^{ij})\sqrt{|X_{m-1}^{ij}|^2+4\tau},\\
        Y_m^j=\frac{1}{2}(X_{m-1}^i+X_{m-1}^j)
                      -\mathrm{sgn}(X_{m-1}^{ij})\sqrt{|X_{m-1}^{ij}|^2+4\tau}.
\end{split}
\]
\item
\[
     X_m^i=Y_{m}^i-\tau  Y_{m}^i+\sqrt{\frac{\tau}{N}}z^i,~~
       X_m^j=Y^j(t_m)-\tau  Y_m^j+\sqrt{\frac{\tau}{N}}z^j.
\]
\end{itemize}
Here, $z^i,z^j\sim \mathcal{N}(0, 1)$.

In Fig. \ref{fig:dysonbmDivision}, we show the numerical results using RBM-1 in section \ref{sec:firstrandom}. The initial data (setting $t=0$ in \eqref{eq:analyticaldis}) are sampled using the Metropolis Hastings algorithm \cite{hastings1970monte}. We plot the results at $t=0.5$ and $t=5$. The number of particles is $N=10^5$ while we use $\tau=10^{-3}$ for time step. As can be seen, RBM-1 can successfully recover the evolution of distribution and the equilibrium semicircle law \eqref{eq:semicircle}, as desired.
In Fig. \ref{fig:dysonbm}, the results of RBM-r are shown. Again, we take $N=10^5$ and $\tau=10^{-3}$. 
Within one iteration,  the same splitting scheme above is used. 
We find that RBM-r indeed has comparable results with RBM-1. 
Though RBM-r seemingly cannot simulate the dynamics of the distributions, $N/2$ iterations in fact has comparable behavior for time $\tau$. This interesting observation confirms that  RBM-r can capture the dynamics for some examples.

Since both stochastic algorithms give similar behaviors, in later examples, we only use one of them to implement for each example. If we care more about the dynamical behavior, we use RBM-1 (see the two examples in Section \ref{sec:social}). Otherwise, we use RBM-r. (In fact, the two algorithms do not show significant difference, even for evolutional problems.)

\begin{figure}
\begin{center}
	\includegraphics[width=1\textwidth]{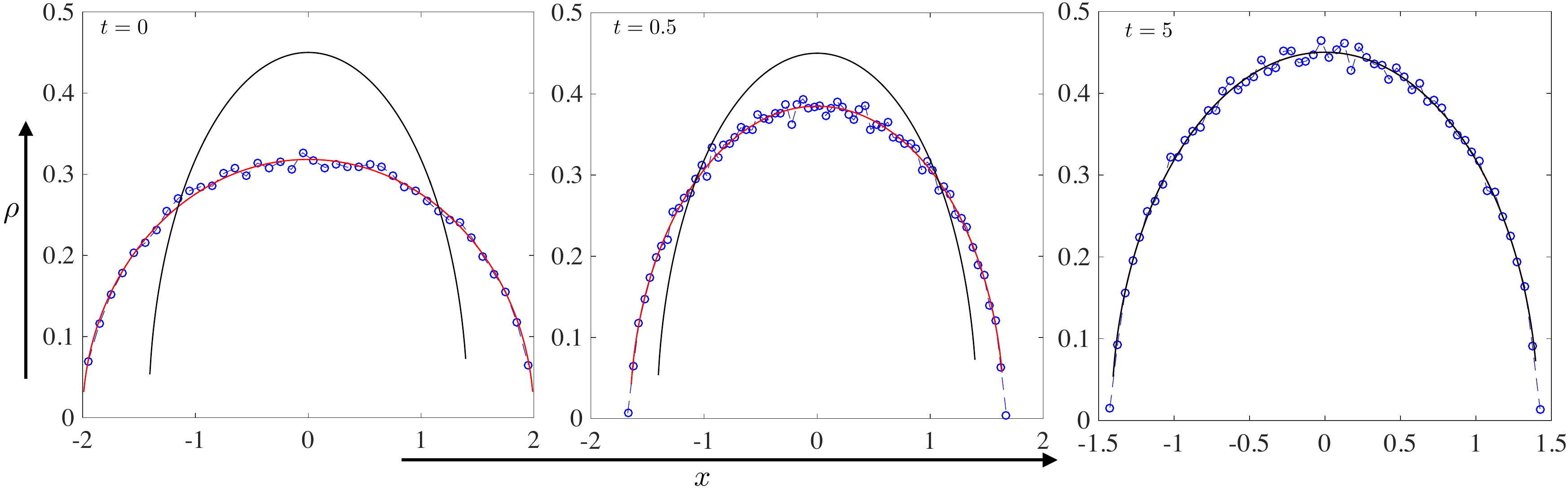}
\end{center}
\caption{The RBM-r simulation of Dyson Brownian motion. The 'time' is regarded as $\tau=10^{-3}$ for $N/2$ iterations. The red curve is the density distribution predicted by analytic solution \eqref{eq:analyticaldis}. The black curve is the equilibrium semicircle law \eqref{eq:semicircle}. }
\label{fig:dysonbm}
\end{figure}

\section{Applications}\label{sec:app}

In this section, we apply RBMs to some examples from physics, social and data sciences. 
Most of these examples do not fit into the theoretic framework of the analysis in section \ref{sec:error}. For example, in section \ref{sec:thomson}, the kernel is singular;  in \ref{sec:social}, the equilibrium still exists though the confining potential does not exist. Even though the analysis in section \ref{sec:error} does not apply for some applications, RBMs turn out to work quite well. For the wealth problem in \ref{sec:social}, we can recover the equilibrium extremely well, though the interaction kernel does not decay and there is no confining potential.
On one hand, the positive results give more supports to the algorithms; on the other hand, for some applications, the stochastic algorithms can be regarded as new models for the underlying problems. 

\subsection{Charged particles on the sphere}\label{sec:thomson}

The traditional Thomson problem is to determine the stable configuration of $N$ electrons on a sphere.
When $N$ becomes large, this could lead to the so-called spherical crystals (\cite{bowick2002,yao2016,yao2017}). The configuration may have some meta-states (local minimizers of the energy surface). When the number of particles is large, the spherical crystals have defects due to the topology of the sphere \cite{yao2016,yao2017}. 

In the $N\to\infty$ limit, hopefully, we will have a continuous distribution of charges on the sphere $\rho(\cdot)$. The problem then corresponds to determining $\rho$ such that the energy 
\begin{gather}
E(\rho)=\frac{1}{2}\iint_{S\times S}\frac{1}{|x-y|}\rho(x)\rho(y)\,dS_xdS_y
\end{gather}
is minimized. It is unclear how the energies corresponding to local minimizers are distributed (if there are any).

Regarding charges with surface densities, let us make a mathematical remark. Suppose $S$ is a surface that divides the whole space $\mathbb{R}^d$ into two halves. Assume there is a continuous distribution of charges on $S$ with density $\rho$. Let $\varphi^{\pm}(x)$ be the limits of the potential on the two sides of $S$, and $\varphi(x):=\varphi^+(x)+\varphi^-(x)$. Then, one has
\begin{gather}\label{eq:fraconS}
(-\Delta)_S^{1/2}\varphi+s(\varphi)=\rho,
\end{gather}
where $(-\Delta)_S^{1/2}$ is the $1/2$ fractional Laplacian on $S$ and $s(\varphi)$ is some pseudo-differential operators with a symbol of degree lower than $1$. In other words, to the leading order, the $1/2$ fractional Laplacian of $\varphi$ equals $\rho$. In the case that $S$ is a plane or a circle in 2D plane, $s(\varphi)=0$. In general, $s(\varphi)\neq 0$. In fact, by the jump condition of electric fields, 
\begin{gather}\label{eq:jumpcondi}
\rho=E^+\cdot n^++E^-\cdot n^-=-\frac{\partial\varphi^+}{\partial n^+}-\frac{\partial\varphi^-}{\partial n^-}.
\end{gather}
It is well-known that the Dirichlet to Neumann operator $\mathcal{L}$ is related to the $1/2$-fractional Laplacian
by
\begin{gather}\label{eq:reducetofracLap}
\mathcal{L}\varphi=-(-\Delta)_S^{1/2}f+r(f),
\end{gather}
where $n$ is the normal vector pointing into the side where the harmonic extension is performed and $r(f)$ includes lower terms. In the case that $S$ is a plane $r(f)=0$ \cite{caffarelli2007}. (In fact, if $S$ is the base of cylinders, $\mathcal{L}=-(-\Delta)_S^{1/2}$ as well. See \cite{cabre2010}. )
 In the case of the unit circle in 2D plane, one can refer to \cite{de2017fractional}. For spheres in higher dimensions, $s\neq 0$. With \eqref{eq:reducetofracLap} and \eqref{eq:jumpcondi}, \eqref{eq:fraconS} follows. 

     Interacting particle systems on the sphere can be realized experimentally by beads on water droplets immersed in oil  \cite{bausch2003}. By adjusting the environmental solution, the interacting kernel $K(\cdot)$ can also be changed, which does not have to be the Coulomb interaction. For such systems, the particles clearly have heat exchange with the enviroment so that the interacting particle system may be described by certain Langevin equations on the sphere:
\begin{gather}
\begin{split}
& dX^i=V^i\,dt,\\
& m dV^i=-\gamma V^i\,dt+\mathbb{P}_S\left(\frac{1}{N-1}\sum_{j\neq i}F(X^i-X^j) \right)\,dt
+\sqrt{2D} \,dB_S^i
\end{split}
\end{gather}
Here, $\mathbb{P}_S$ is the projection onto the tangent space of the sphere while $B_S^i$ is the spherical Brownian motion to guarantee that the particle stays on the sphere. For theories of SDEs on manifolds, one may refer
to \cite{hsu02}. 
Here, $D$ and $\gamma$ must be related as in the classical fluctuation-dissipation theorem \cite{callen1951}. 

In the overdamped limit and with suitable scaling, we then have interacting particle system on the sphere ($D_1=D/\gamma^2$)
\begin{gather}
dX^i=\mathbb{P}_S\left(\frac{1}{N-1}\sum_{j\neq i}F(X^i-X^j)\right)\,dt+\sqrt{2D_1} dB_S^i
\end{gather}
Numerically discretizing such SDEs on the sphere is an interesting topic which we leave for the future. In this work, 
we consider the Coulomb interaction with $\sigma :=\sqrt{2D_1}=0$, and use RBM-r as the stochastic strategy.  The following simple scheme for the SDE on the sphere is then adopted.
\begin{itemize}
\item Randomly picking two indices. Then, solve the following for time $t\in [t_{m-1}, t_m)$
\begin{gather}
dX^i=\sum_{j: I(i,j)=1}\frac{X^i-X^j}{|X^i-X^j|^3}\,dt,
\end{gather}
where $I(i,j)=1$ means that $i,j$ are in the same batch. This can be solved analytically. In particular, defining 
\begin{gather}
\hat{v}_m=(X_{m-1}^i-X_{m-1}^j)/|X_{m-1}^i-X_{m-1}^j|,
\end{gather}
one then has:
\begin{gather}
\begin{split}
&X_m^i=(X_{m-1}^i+X_{m-1}^j)+\hat{v}_m(|X_{m-1}^i-X_{m-1}^j|^3+6\tau)^{1/3},\\
&X_m^j=(X_{m-1}^i+X_{m-1}^j)-\hat{v}_m(|X_{m-1}^i-X_{m-1}^j|^3+6\tau)^{1/3}
\end{split}
\end{gather}

\item Project the obtained points back to the sphere by dividing its magnitude.
\end{itemize}
The reason for setting $\sigma=\sqrt{2D_1}=0$ is that we would like to explore energy stable configurations. 
We desire low temperature regime for the ground state. Besides, the stochastic algorithm also introduces randomness so that we still have chance to get out of the local minimizers.

\begin{figure}
\begin{center}
	\includegraphics[width=0.8\textwidth]{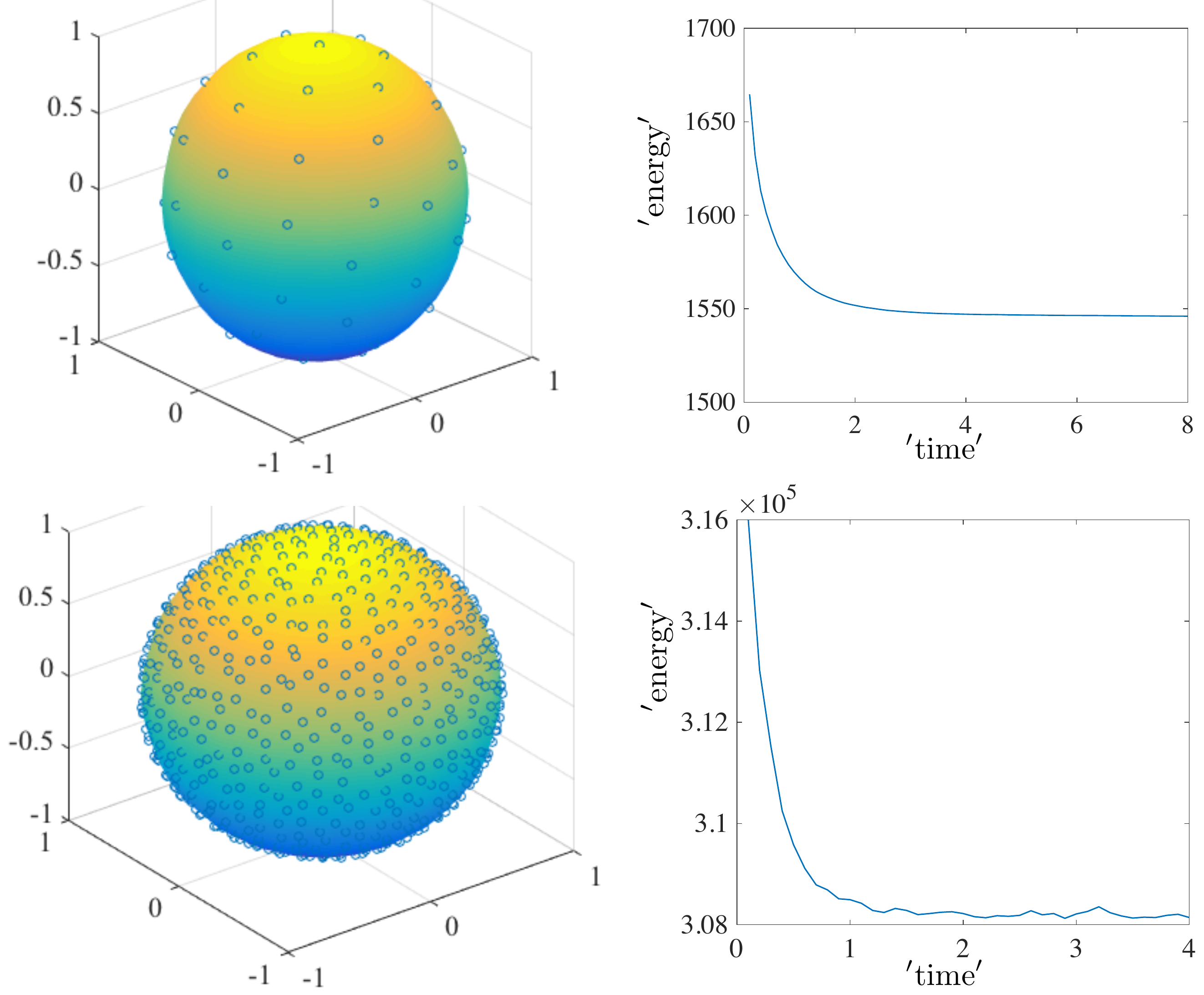}
\end{center}
\caption{Charged particles on the sphere. The first row is for $N=60$ while the second row corresponds to $N=800$. The first column shows the distributions at the end of simulation while the second column shows how the energy changes with 'time'.}
\label{fig:chargesphere}
\end{figure}

To check whether the method can give the desired ground state approximately, we randomly choose initial points on the sphere and run the above stochastic algorithms for enough iterations.
We do many experiments and check whether we always obtain the same final energy level.

 In Fig. \ref{fig:chargesphere}, we show the numerical results in two experiments. 
 The number of particles are chosen as $N=60$ and $N=800$ respectively. The initial points are chosen randomly. The time step is chosen as $\tau=10^{-4}$. As before, we regard the 'time' to be $\tau$ after $N/2$ iterations.
 For $N=60$, we see that in the eventual near stable configuration, each particle has $5$ or $6$ neighbors, and this agrees with the known results by physicists \cite{yao2016,yao2017}. This configuration is quite different from the fullerene $C_{60}$ structure which is induced by the special properties of Carbon atoms.
For the $N=800$ case, the particles are roughly distributed uniformly. For both figures, there is only one stable energy level during the whole process. This means the system was only trapped in the final stable configuration.

 To check whether there are other possible stable energy configurations, we collect in Fig. \ref{fig:chargesphere1} the energies for $N=800$ after $3*(N/2)/\tau$ iterations ($T=3$) in $25$ experiments.  The simulation shows that one can find the ground state of the configuration almost surely using the stochastic algorithm and it is not easy to be trapped in local minimizers, if there are any.
  As studied by the physicists, there are many energy levels for the Thomson sphere. However, the numerical results here seem to suggest that the stochastic algorithms can obtain the ground state with high probability and the local minimizers of the energy landscape probably has small energy barriers. 
 Maybe, some interesting phenomena happen for large $N$'s which needs further investigation.
 
 \begin{figure}
\begin{center}
	\includegraphics[width=0.4\textwidth]{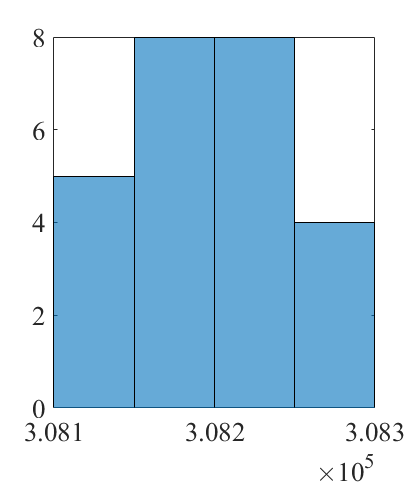}
\end{center}
\caption{Energy statistics. This histogram shows the terminal energy for 25 experiments. The terminal 'time' is $T=3$.}
\label{fig:chargesphere1}
\end{figure}

\subsection{Two examples from economics and social science}\label{sec:social}

In this section, we apply RBM-1 for two important models in social sciences, namely the evolution of wealth \cite{degond2014} and opinion dynamics \cite{motsch2014}. The obtained stochastic processes not only are algorithms for the original models, but also can be viewed as new models which consider the fact that only a few individuals commute during a short time.

\subsubsection{Stochastic dynamics of wealth}\label{sec:wealth}

We consider the model proposed by Degond et al. \cite{degond2014}, which tries to understand
the evolution of $N$ market agents with two attributes: the economic configuration $X^i$ and its wealth $Y^j$.
\begin{gather}
\begin{split}
& \dot{X}^i=V(X^i, Y^i),\\
& dY^i=-\frac{1}{N-1}\sum_{k:k\neq i}\xi_{ik}\Psi(|X^i-X^k|)\partial_y\phi(Y^i-Y^k)\,dt+\sqrt{2D} Y^i dB^i.
\end{split}
\end{gather}
The first equation describes the evolution of the economic configuration, which is driven by the local Nash equilibrium and it is related to mean-field games \cite{degond2014b,lasry2007}. The second equation describes the evolution of the wealth, which contains two mechanisms: the trading model proposed by Bouchaud and Mezart \cite{bouchaud2000}, and the geometric Brownian motion in finance proposed by Bachelier in 1900 \cite{bachelier1900}. The quantity $\sqrt{2D}$ is the volatility. The function $\phi$ is the trading interaction potential, while $\xi_{ik}\Psi(|X^i-X^k|)$ is the trading frequency. Often one assumes that $\xi_{ik}$ depends on the number of trading agents in the economic neighborhoods of $i$ and $k$:
\begin{gather}
\xi_{ik}=\xi\left(\frac{\rho^{i,\Psi}+\rho^{k,\Psi}}{2}\right),~~\rho^{i,\Psi}=\frac{1}{N-1}\sum_{\ell\neq j}\Psi(|X^{\ell}-X^i|).
\end{gather}
The mean field Fokker-Planck equation is given by
\begin{gather}
\partial_t f+\partial_x(V(x, y)f)+\partial_y(F f)=D\partial_{yy}(y^2 f)
\end{gather}
where 
\begin{multline}
F(x, y, t)=-\int_{x'\ge 0, y'\ge 0}\xi\left(\frac{1}{2}(\rho^{\Psi}(x,t)+\rho^{\Psi}(x', t))\right)\Psi(|x-x'|) \\
\times \partial_y\phi(y-y')f(x',y',t)\,dx' dy',
\end{multline}
and 
\[
\rho^{\Psi}(x, t)=\int_{x'>0, y'>0}\Psi(|x-x'|)f(x', y', t)dx' dy'.
\]

Now, if one considers the homogeneous case where the wealth dynamics is independent of the position in the economic configuration space, then $\Psi$ is a constant. In this case, the dynamics is reduced to the interacting particle system, except that one has {\it multiplicative} noise
\begin{gather}\label{eq:wealth}
 dY^i=-\frac{\kappa}{N-1} \sum_{k:k\neq i}\partial_y\phi(Y^i-Y^k)\,dt+\sqrt{2D} Y^i dB^i,
\end{gather}
where $\kappa:=\Psi\xi\left(\frac{1}{2}(\rho^{i\Psi}+\rho^{k,\Psi})\right)$ is now a constant. The mean field equation is now given by
\begin{gather}
\partial_t\rho(y)=-\partial_y(F(y)\rho(y))+D\partial_{yy}(y^2\rho(y)),
\end{gather}
where
\[
F(y)=-\kappa\int_{y\ge 0}\partial_y\phi(y-y')\rho(y')dy'.
\]
The equilibrium distribution is given by
\[
\rho_{\infty}(y)\propto \exp\left(-\frac{\alpha(y)}{D}\right),
\]
where $\alpha$ satisfies
\[
\partial_y\alpha(y)=-\frac{1}{y^2}F(y)+\frac{2D}{y}.
\]

We now apply RBM-1 with $p=2$ to \eqref{eq:wealth} and have
for $t\in [t_{m-1}, t_m)$
\begin{gather}\label{eq:wealthrandom}
dY^i=-\kappa \partial_y\phi(Y^i-Y^{\theta})\,dt+\sqrt{2D} Y^i dB^i,~i=1,\ldots, N,
\end{gather}
where $\theta$ is a random index that is grouped with $i$ in the random division.
In some sense, the stochastic dynamics described by this algorithm can model what is happening in the real world: each agent only trades with a small number of random agents at a time. Then, after some time interval, the agents trade with others. Hence, \eqref{eq:wealthrandom} is not just an algorithm but also it can be viewed as a {\it new model}.

For numerical test, choose the quadratic trading interaction as in \cite[section 3.4]{degond2014}
\[
\phi(y)=\frac{1}{2}y^2.
\]
This interaction function may not be practical as it increases with $y$ (intuitively, as $y\to\infty$, it should go to zero). The good thing is that with this interaction function, one can find the equilibrium distribution of wealth for \eqref{eq:wealth} as
\begin{gather}\label{eq:rhoinfty}
\rho_{\infty}(y)=\frac{(\kappa \eta/D)^{\kappa/D+1}}{\Gamma(\kappa/D+1)}y^{-(2+\kappa/D)}\exp\left(-\frac{\kappa \eta}{Dy}\right)1_{y>0} .
\end{gather}
This distribution is the inverse Gamma distribution and agrees with the Pareto power law for large $y$.
Here, $\eta$ is the mean wealth. 

Now, we take $\kappa=D=1$ and consider the random dynamics \eqref{eq:wealthrandom}:
\[
dY^i=-  (Y^i-Y^{\theta})\,dt+\sqrt{2} Y^i dB^i,~i=1,\ldots, N .
\]

\begin{figure}
\begin{center}
	\includegraphics[width=0.8\textwidth]{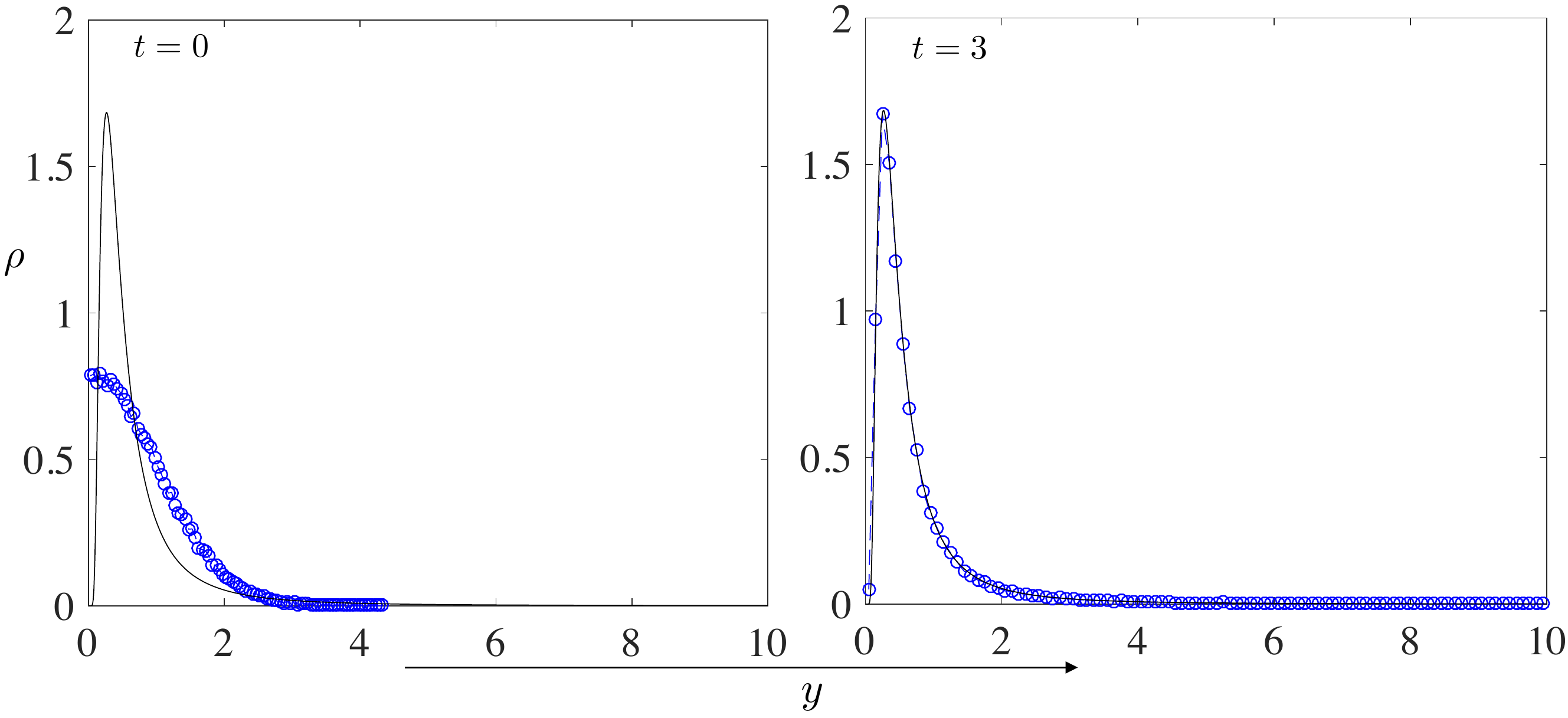}
\end{center}
\caption{Wealth distribution for stochastic dynamics}
\label{fig:wealth}
\end{figure}

In Fig. \ref{fig:wealth}, we plot the empirical distribution of the wealth with $N=10^5$ agents. We choose $\tau=10^{-3}$ and do the simulation to $T=3$. The SDE again is solved by the splitting strategy. Note that the splitting scheme preserves the mean wealth. For the test, we choose initial data as $X^i=|Y^i|$ with $Y^i\sim \mathcal{N}(0, 1)$.
The reference curve is \eqref{eq:rhoinfty} with 
\begin{gather}
\eta=\frac{\sqrt{2}}{\sqrt{\pi}}.
\end{gather}
Clearly, the numerical results agree perfectly with the expected wealth distribution.

\subsubsection{Stochastic opinion dynamics}
In this section, we consider some stochastic revisions of the opinion dynamics in \cite{motsch2014}, where the following two models are mentioned
\begin{subequations}
\begin{gather}\label{eq:opinion1}
\frac{d}{dt}X^i=\alpha \frac{1}{N-1}\sum_{j\neq i}\phi(|X^j-X^i|) (X^j-X^i)
\end{gather}
and
\begin{gather}\label{eq:opinion2}
\frac{d}{dt}X^i=\alpha \sum_{j\neq i}\frac{\phi(|X^j-X^i|)}{\sum_{k}\phi(|X^k-X^i|)} (X^j-X^i).
\end{gather}
\end{subequations}
Here, $\phi$ is called the influence function.
These models are introduced for the emergence of consensus of opinions. Note that in the original
model, the prefactor is $1/N$ and we use $1/(N-1)$ here. The effect of chis change is minor.

Here, \eqref{eq:opinion2} is not convenient for stochastic algorithms because of the denominator. Instead, we consider RBM-1 applied to \eqref{eq:opinion1}, and have the following:
\begin{gather}
dX^i=\alpha \phi(|X^{\theta}-X^i|) (X^{\theta}-X^i)\,dt
+\epsilon_N dB^i,
\end{gather}
where $\theta$, again as before, is a random index that is fixed for $t\in [t_{m-1}, t_m)$.  If $\epsilon_N=0$, this is the stochastic algorithm of \eqref{eq:opinion1} directly.
The parameter $\epsilon_N$ is to represent the random fluctuation on its opinions. If the number of agents is large, we believe this should be small.
In fact, this stochastic algorithm seems closer to what is happening in the world: each person only talks to one or few people at a time. Combining its previous opinion, it forms a new opinion.

Let us consider the influence function
\begin{gather}
\phi(r)=\chi_{[0, 1]}
\end{gather}
as in \cite[Section 3.3]{motsch2014}.

\begin{figure}
\begin{center}
	\includegraphics[width=1\textwidth]{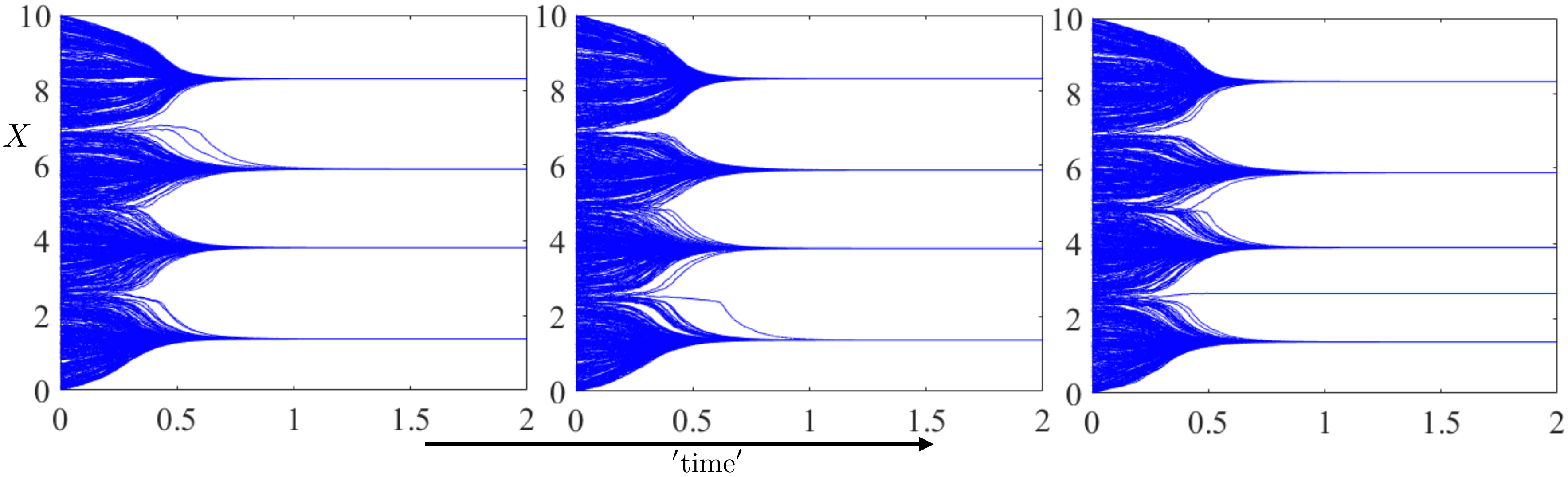}
\end{center}
\caption{Stochastic opinion dynamics versus time. Plots of three experiments of the stochastic dynamics with the same initial data. No Brownian motion ($\epsilon_N=0$).}
\label{fig:opinion}
\end{figure}
In Fig. \ref{fig:opinion},  we show the numerical results. The figure shows the results of three experiments with the same initial data. We choose $N=10^3$ and $\tau=10^{-4}$, $\alpha=40$. 
Clearly, the final stationary consensus is sensitive with respect to initial distribution. With the same initial data, though the dynamics is stochastic, the main behavior is the same for the three experiments.  There are four main clusters of consensus. However, interestingly, in some experiments (like the third picture), there are may be some individuals that do not belong to any cluster, which seems to be the case in real world: some individuals are isolated at the early stage, and after the main clusters of consensus form, they are not affected by these groups since they are so different. The randomness introduced by the algorithms does not quite affect the main clusters of consensus, and only a few individuals might behave differently due to the randomness. After certain time, when the clusters of consensus are formed, the randomness does not play any roles any more: the individuals only talk to members in their own clusters.

In Fig. \ref{fig:opinionBM},  we show the numerical results for the stochastic opinion dynamics with Brownian motion
$\epsilon_N=\frac{1}{N^{1/3}}$. The $N$, $\tau$, $\alpha$ parameters are the same.
 The evolution of clusters of consensus is roughly the same with or without the Brownian motion. However, Brownian motions indeed introduce fluctuation of opinions within the clusters. This means that the fluctuation is not very important when the main clusters of opinions are formed.

\begin{figure}
\begin{center}
	\includegraphics[width=0.8\textwidth]{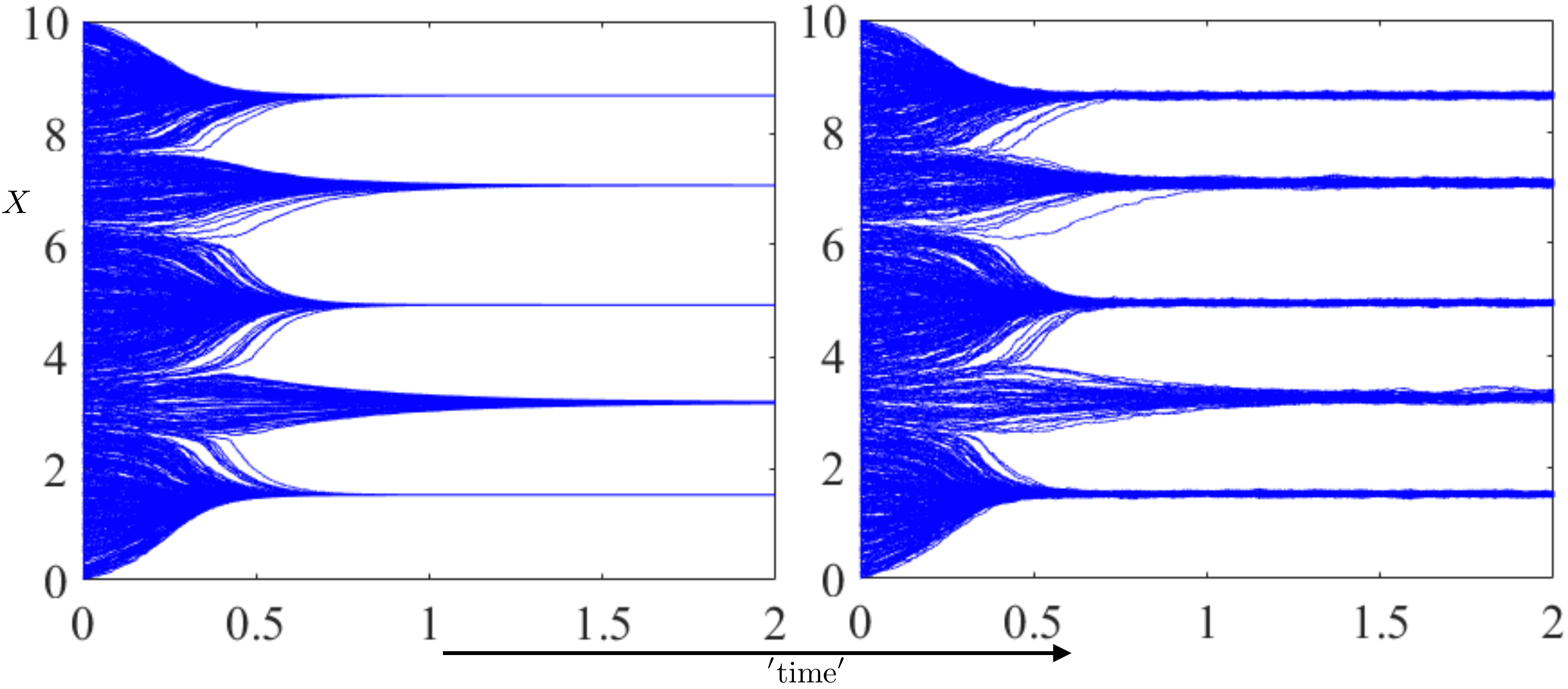}
\end{center}
\caption{Opinion dynamics versus time. The first figure has no Brownian motion. The second is with $\epsilon_N=N^{-1/3}$. The two figures are with the same initial data.}
\label{fig:opinionBM}
\end{figure}

\subsection{Clustering through interacting particle system}\label{sec:cluster}

In this section, we consider using the interacting particle system \eqref{eq:interactingps} for clustering, as discussed in \cite{fang2018emergent}. 
The idea is like this: consider $N$ particles with a given adjacency matrix
$A=(a_{ij})$, $a_{ij}\ge 0$. Then, we construct the interacting particle system as
\begin{gather}
\frac{d}{dt}X^i= \frac{\alpha}{N-1}\sum_{j: j\neq i}(a_{ij}-\beta)(X^j-X^i)
\end{gather}
where $\alpha$ and $\beta$ are some parameters. This is designed such that the particles with positive $a_{ij}-\beta$ attract with each other so that they tend to gather together, while those with negative $a_{ij}-\beta$ repel each other so that they separate. The hoping is that the intrinsic clusters will emerge automatically.

With the RBM-r, the computational cost is significantly reduced and this then becomes a practical method.
Each time, we pick a random set $\mathcal{C}$ that contains $p=2$ elements. The dynamics we consider is then
\begin{gather}\label{eq:dynamicscluster}
\frac{d}{dt}X^i=\alpha (a_{i\theta}-\beta)(X^{\theta}-X^i)
\end{gather}
where $\theta$ is again the random index that is constant on $[t_{m-1}, t_m)$. Clearly, the random batch can be picked from those with nonzeros of $a_{ij}$ only to improve the efficiency.

The update formula is 
\[
\begin{split}
X_m^i=[(X_{m-1}^i+X_{m-1}^j)+(X_{m-1}^i-X_{m-1}^j)\exp(-2\alpha (a_{ij}-\beta)\tau)]/2;\\
X_m^j=[(X_{m-1}^i+X_{m-1}^j)-(X_{m-1}^i-X_{m-1}^j)\exp(-2\alpha  (a_{ij}-\beta)\tau)]/2 .
\end{split}
\]

\begin{figure}
\begin{center}
	\includegraphics[width=0.45\textwidth]{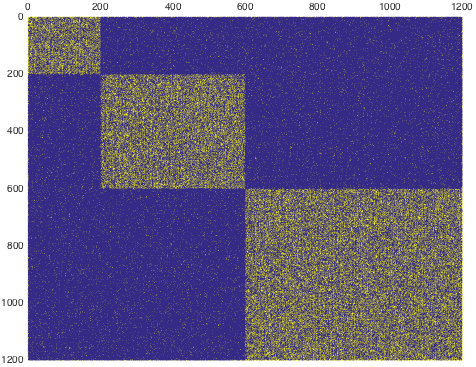}
\end{center}
\caption{Clustering: an adjacent matrix for the stochastic block model}
\label{fig:clustermat}
\end{figure}

\subsubsection{Cluster for stochastic block model}\label{eq:cluster1}

Let us consider the stochastic block model (\cite{holland1983,damle2016}). The model is like this: suppose there are $k$ clusters. For $i, j$ in the same cluster, $\mathbb{P}(a_{ij}=1)=p$ and $\mathbb{P}(a_{ij}=0)=1-p$. Otherwise, $\mathbb{P}(a_{ij}=1)=q$ and $\mathbb{P}(a_{ij}=0)=1-q$. The entries $a_{ij}$ are assumed to be independent. 
We assume that we only know the adjacent matrix in one experiment (if we know the matrices for several experiments, we can then combine them to get more accurate clusters). Clearly, the adjacent matrix is noisy. We are going to test whether or not we can still recover the clusters using the noisy adjacent matrix.

In Fig. \ref{fig:clustermat}, we show the adjacent matrix from one experiment. In this example, we have $N=1200$ particles. The first $200$ particles are designed to be in the first cluster, the next $400$ are in the second cluster and the last $600$ are in the third cluster. The probabilities are chosen as 
\[
p=0.7,~~q=0.3.
\] 
In the matrix, yellow dots mean $1$ while blue dots mean $0$.
As expected,  most of the entries in the off-diagonal blocks are $0$ with some yellow 'dust' scattered inside them.  
Most of the entries in the diagonal blocks are $1$ with blue dots inside them.

\begin{figure}
\begin{center}
	\includegraphics[width=0.9\textwidth]{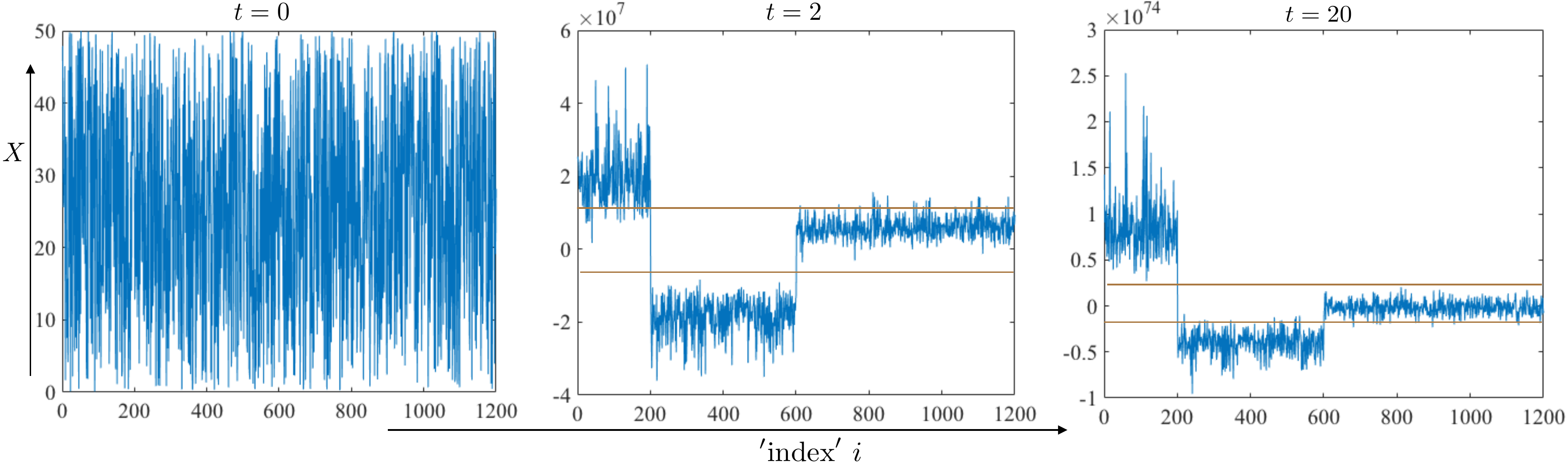}
\end{center}
\caption{Clustering using a noisy adjacent matrix from the stochastic block model. The plots show the positions of the particles at different 'time' points. The desired clusters can be recovered with few mistakes.}
\label{fig:cluster1}
\end{figure}
In the experiment, we set $\beta=\frac{1}{2}$ and $\alpha=40$, $\tau=10^{-3}$.
We initialize their positions randomly on $[0, 50]$. 
The numerical results in an experiment are shown in Fig. \ref{fig:cluster1}. 
 Again the "time" is regarded as $\tau$ after $N/2$ iterations.
 From the figure, it is clear that the clusters can be recovered correctly though the adjacent matrix is noisy. 

\subsubsection{Reordering for sparse matrix}

As another example, let us consider reordering sparse matrices as a byproduct of the clustering.
The point is that large $a_{ij}$ entry tends to group the two indices together. 
If we use the terminal $X_i$'s to sort, the reordered matrix can have large entries near the diagonal. If there are several distinct clusters, we will then have diagonal block matrix.

Consider the matrix given by
\[
A_1=BB^T+I,
\]
where $B$ is the 'west0479' matrix, which is a sparse matrix in the standard database of  MATLAB. Consequently, $A_1$ is a sparse matrix. Since matrix $A_1$ can have negative entries, we define $A=(a_{ij})$ with $a_{ij}=|A_1(i,j)|$ to get a suitable adjacent matrix. Since $A$ is sparse, we do sampling over the nonzero entries only.
\begin{figure}
\begin{center}
	\includegraphics[width=0.9\textwidth]{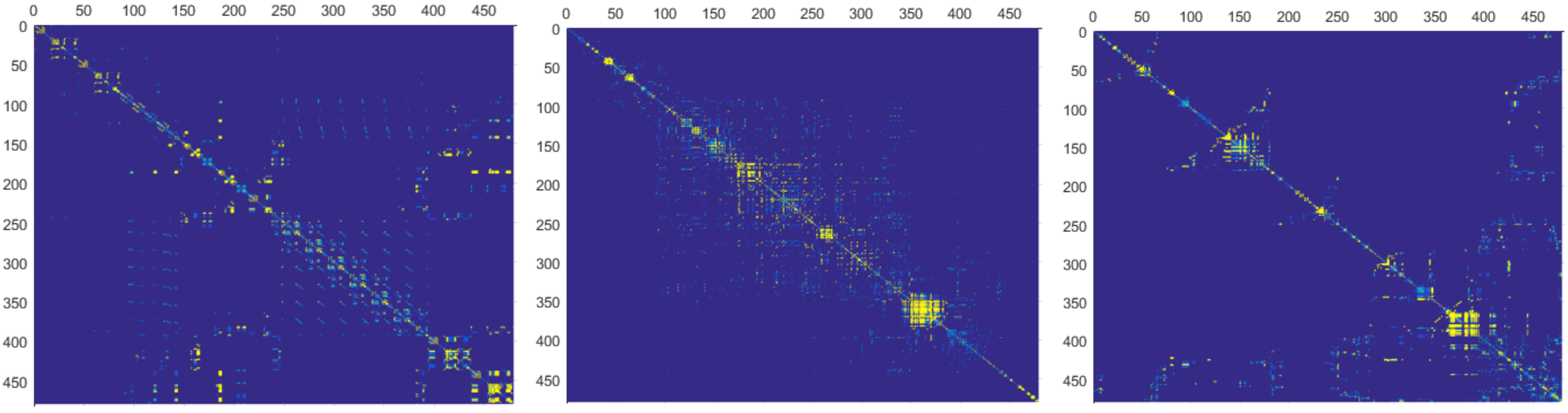}
\end{center}
\caption{A sparse adjacent matrix: clustering and reordering}
\label{fig:cluster3}
\end{figure}

Fig. \ref{fig:cluster3} shows the matrices $\min(P, 10)$ ("$\min$" here means the entry-wise minimum), where the meaning of $P$ is as follows: in the first figure, $P$ is $A$; in the second figure, $P$ is the reordered matrix using our strategy at $T=5$; $P$ in the third figure is the matrix with approximated minimal degree ordering.
Clearly, our strategy gathers all the big entries near the diagonal, which means particles with strong interactions indeed form clusters.

The ordering resulted from our clustering strategy works not that good as the minimal degree ordering using the criterion from sparse matrix theory. For example, in the Cholesky decomposition, there are $37896$ nonzero entries. Meanwhile, by the approximated minimal degree ordering,  there are $14493$ nonzero entries in the Cholesky decomposition.  However, the ordering using clustering gathers big entries near diagonals, which can yield better stability and may be advantageous for some applications.

\section{Conclusions}\label{sec:conclusion}

We have developed Random Batch Methods for interacting particle systems with large number of particles and they reduce the computational cost significantly for $N (N\gg 1)$ particles from $O(N^J)$ to $O(N)$ per time step. For RBM-1, the method without replacement, we have given a particle number independent error estimate under some special interactions. We have applied these methods to some representative problems of binary interactions in math, physics, social science and data science, and numerical results have supported our theory and expectations.
The random algorithms are powerful for systems with large number of individuals and high dimensions. 
 
As well accepted, in stochastic gradient descent, adding momentum could help to find flatter minimizers and improve results. In other words, the Langevin dynamics seems better for optimization and sampling \cite{lelievre2016,cheng2017}. Hence, considering the interacting particles with mass might be better for sampling the invariant measure of the nonlinear Fokker-Planck equation. This is left for future research.
Besides, there are many interesting projections ahead, for example,  proof of convergences for more general external and interacting potentials, and for RBMs with replacements.  It is also interesting to develop similar particle methods for the mean field equations, whenever they are available, as was done in \cite{albi2013} but for more general mean field equations.

\section*{Acknowledgement}
S. Jin was partially supported by the NSFC grant No. 31571071.  The work of L. Li was partially sponsored by NSFC 11901389, Shanghai Sailing Program 19YF1421300 and NSFC 11971314. The work of J.-G. Liu was partially supported by KI-Net NSF RNMS11-07444 and NSF DMS-1812573.

\appendix

\section{An error analysis for a Hamiltonian system}\label{sec:Hamil}

In this section, we give a convergence proof of the RBM-1 for the following second order (Hamiltonian) system
\begin{gather}\label{eq:Hamil1}
\begin{split}
& \dot{X}^i=V^i,\\
& \dot{V}^i= b(X^i)+\frac{1}{N-1}\sum_{j: j\neq i}K(X^i-X^j),
\end{split}
\end{gather}
where the dot means time derivative. Note that the variable '$V$' here is reserved for the velocity so we use $b(\cdot)$ to represent the external force field corresponding to $-\nabla V$ in the main text.

Similarly, the processes generated by RBM-1 are determined by the following ODEs for $t\in [t_{m-1}, t_m)$
\begin{gather}\label{eq:Hamilrand}
\begin{split}
&\dot{\tilde{X}}^i=U^i,\\
& \dot{U}^i= b(\tilde{X}^i)+\frac{1}{p-1}\sum_{j\in\mathcal{C}_{q}: j\neq i}K(\tilde{X}^i-\tilde{X}^j),
\end{split}
\end{gather}
where $\mathcal{C}_q$ is the batch in which $i$ lives.

Assume that the initial data are drawn randomly and independently:
\begin{assumption}\label{ass:second}
Systems \eqref{eq:Hamil1} and \eqref{eq:Hamilrand} share the same initial data $X^i(0)=\tilde{X}^i(0)=X_0^i$ and 
$V^i(0)=U^i(0)=V_0^i$, and the initial data $(X_0^i, V_0^i)$ are i.i.d sampled from some common distribution.
\end{assumption}

With the setup, one can show the convergence of RBM-1 similarly.
\begin{theorem}\label{thm:secondconv}
Let Assumption \ref{ass:second} hold. Suppose $b(\cdot)\in C^1(\mathbb{R}^d)$ with $b$ and $\nabla b$ being bounded, and $K$ is bounded and Lipschitz continuous. Then
the RBM-1 converges on $[0, T]$ in the sense that 
\begin{gather}
\begin{split}
\sup_{0\le t\le T}\Big(\mathbb{E}\frac{1}{N}\sum_{j=1}^N(|X^j(t)-\tilde{X}^j(t)|^2+|V^j(t)-U^j(t)|^2)\Big)^{1/2}  &= \\
\sup_{0\le t\le T}\Big(\mathbb{E}(|X^i(t)-\tilde{X}^i(t)|^2+|V^i(t)-U^i(t)|^2)\Big)^{1/2}  &\le C(T)\sqrt{\frac{\tau}{p-1}},
\end{split}
\end{gather}
for all $i\in \{1,\ldots, N\}$ and some $C(T)>0$ independent of $N$.
\end{theorem}

To prove this, one similarly introduces the error processes:
\begin{gather}
\begin{split}
& Z^i=\tilde{X}^i-X^i,\\
& W^i=U^i-V^i
\end{split}
\end{gather}

Clearly, the error processes satisfy for $t\in [t_{m-1}, t_m)$ that
\begin{gather}
\begin{split}
& \dot{Z}^i =W^i,\\
& \dot{W}^i = \Big( b(\tilde{X}^i)-b(X^i) \Big)
+\frac{1}{N-1}\sum_{j: j\neq i}\delta K_{ij}(t)
+\chi_{m,i}(\tilde{X}),
\end{split}
\end{gather}
where again
\begin{gather}
\delta K_{ij}=K(\tilde{X}^i-\tilde{X}^j)-K(X^i-X^j),
\end{gather}
and the error of the interaction force is given by
\begin{gather}
\chi_{m,i}(x):=\frac{1}{p-1}\sum_{j\in\mathcal{C}_{\theta},j\neq i}K(x^i-x^j)
-\frac{1}{N-1}\sum_{j: j\neq i} K(x^i-x^j).
\end{gather}
The notation $\mathcal{C}_{\theta}$ is used as we want to emphasize that it is a random set over all experiments instead of the realization in one experiment (for which we use $\mathcal{C}_q$).

The following lemmas are some preparation for the proof of the convergence, which are analogies of the lemmas in Section \ref{sec:error} but are much easier due to the absence of Brownian motions.
\begin{lemma}
For any $q>1$, the $q$-moments are bounded almost surely. More precisely, there exists $C_q(T)>0$ such that it holds almost surely that
\begin{gather}
\sup_{0\le t\le T}(|\tilde{X}^i(t)|^q+|U^i(t)|^q) \le C_q(T),
\end{gather}
and 
\begin{gather}
\sup_{0\le t\le T}(|X^i(t)|^q+|V^i(t)|^q)\le C_q(T).
\end{gather}
\end{lemma}
The proof is very straightforward for which we omit the details. For example, regarding system \eqref{eq:Hamilrand}, it holds
almost surely that for $t\in [t_{m-1}, t_m]$:
\begin{gather*}
\begin{split}
\frac{d}{dt}(|\tilde{X}^i|^q+|U^i|^q) &=q|\tilde{X}^i|^{q-2}\tilde{X}^i\cdot U^i+ q|U^i|^{q-2} U^i\cdot\left(b(\tilde{X}^i)+\frac{1}{p-1}\sum_{j\in\mathcal{C}_{\theta}: j\neq i}K(\tilde{X}^i-\tilde{X}^j)\right)\\
&\le C|\tilde{X}^i|^{q-1}|U^i|+C|U^i|^{q-1}\\
&\le C_1(|\tilde{X}^i|^q+|U^i|^q)+C_2.
\end{split}
\end{gather*}
This will give the uniform bound of $|\tilde{X}^i|^q+|U^i|^q$ on $[0, T]$.

With the moments bound, it is easy to conclude the following for which we omit the proof.
\begin{lemma}\label{lmm:uniformincrementsecond}
For all $t\in [t_{m-1}, t_m)$,
\begin{gather}
\|\tilde{X}^i(t)-\tilde{X}^i(t_{m-1})\|+\|U^i(t)-U^i(t_{m-1})\|\le C(T)\tau,
\end{gather}
where $\|\cdot\|$ again means the $L^2(\Omega)$ norm.
Also, almost surely, it holds that
\begin{gather}
|U^i(t)-U^i(t_{m-1})|+|V^i(t)-V^i(t_{m-1})|\le C(T)\tau.
\end{gather}
\end{lemma}

\begin{proof}[Proof of Theorem \ref{thm:secondconv}]
Define again
\begin{gather}
u(t)=\mathbb{E}\frac{1}{N}\sum_{j=1}^N (|Z^j|^2+|W^j|^2).
\end{gather}
Due to exchangeability, 
\begin{gather}
u(t)=\mathbb{E} (|Z^i|^2+|W^i|^2),~\forall i\in \{1, \ldots, N\}.
\end{gather}

For $t\in [t_{m-1}, t_m)$, taking time derivative, one has
\begin{multline*}
\frac{d}{dt}u=2\mathbb{E}\left(Z^i\cdot W^i+W^i\cdot(b(\tilde{X}^i)-b(X^i))
+W^i\cdot \frac{1}{N-1}\sum_{j: j\neq i}\delta K_{ij}\right)\\
+2\mathbb{E} W^i(t)\cdot\chi_{m,i}(Y(t)).
\end{multline*}
Since $K$ is Lipschitz continuous, 
\[
\begin{split}
& \mathbb{E}\left(Z^i\cdot W^i++W^i\cdot(b(\tilde{X}^i)-b(X^i))
+W^i\cdot \frac{1}{N-1}\sum_{j: j\neq i}\delta K_{ij} \right)\\
   & \le C\mathbb{E}(|Z^i| |W^i|)
+C\frac{1}{N-1}\sum_{j: j\neq i}\mathbb{E}|W^i||Z^j| \\
& \le Cu(t).
\end{split}
\]
The second inequality above is due to exchangeability.

Now, one can do the same trick as in Section \ref{sec:error}:
\begin{multline}
\mathbb{E} W^i(t)\cdot\chi_{m,i}(\tilde{X}(t)) =\mathbb{E}W^i(t_{m-1})\cdot\chi_{m,i}(\tilde{X}(t_{m-1}))
+\mathbb{E}W^i(t_{m-1})\cdot (\chi_{m,i}(\tilde{X}(t))-\chi_{m,i}(\tilde{X}(t_{m-1})))\\
+\mathbb{E}(W^i(t)-W^i(t_{m-1}))\cdot \chi_{m,i}(X(t))\\
+\mathbb{E}(W^i(t)-W^i(t_{m-1}))\cdot (\chi_{m,i}(\tilde{X}(t))-\chi_{m,i}(X(t)))
=: I_1+I_2+I_3+I_4.
\end{multline}

Similarly, $I_1=0$, due to Lemma \ref{lmm:averagefunc}. Also,
\[
I_2\le \|W^i(t_{m-1})\| \left\|\frac{1}{p-1}\sum_{j\in \mathcal{C}_{\theta}, j\neq i} |\delta K^{ij}|
+\frac{1}{N-1}\sum_{j: j\neq i} |\delta K^{ij}| \right\|,
\]
where
\[
\delta K^{ij}:=K(\tilde{X}^i(t)-\tilde{X}^j(t))-K(\tilde{X}^i(t_{m-1})-\tilde{X}^j(t_{m-1})).
\]
Clearly, due to the Lipschitz continuity of $K$ and the equation of $\tilde{X}^i$, it holds
\[
|\delta K^{ij}|\le C\left(\int_{t_{m-1}}^t |U^i(s)|\,ds
+\int_{t_{m-1}}^t |U^j(s)|\,ds \right)
\le C(|U^i(t_{m-1})|\tau+|U^j(t_{m-1})|\tau+\tau^2 ),
\]
where the almost sure bound of $U^i(t)-U^i(t_{m-1})$ in Lemma \ref{lmm:uniformincrementsecond} has been used.
Applying Lemma \ref{lmm:normofrandomsum} regarding the norm of random sum, one therefore obtains
\[
I_2\le C\|W^i(t_{m-1})\|\tau\le C(\|W^i(t)\|\tau+\tau^2).
\]

Similar as in Section \ref{sec:error}, one has
\[
I_3\le C\sqrt{u}\tau+\|\Lambda_i\|_{\infty}\left(\frac{1}{p-1}-\frac{1}{N-1}\right)\tau.
\]

Similarly as $I_2$, one can control $I_4$ by
\[
I_4\le \|W^i(t)-W^i(t_{m-1})\| \left\|\frac{1}{p-1}\sum_{j\in \mathcal{C}_{\theta}, j\neq i} |\delta K_{ij}|
+\frac{1}{N-1}\sum_{j: j\neq i} |\delta K_{ij}| \right\|,
\]
where
\[
\delta K_{ij}(t)=K(\tilde{X}^i(t)-\tilde{X}^j(t))-K(X^i(t)-X^j(t)).
\]
Using the Lipschitz continuity of $K$ again and the equation of $X^i, \tilde{X}^i$, one has
\begin{multline*}
|\delta K_{ij}|\le  |\tilde{X}^i-X^i|(t_{m-1})+|U^i(t_{m-1})-V^i(t_{m-1})|\tau\\
+|\tilde{X}^j-X^j|(t_{m-1})+|U^j(t_{m-1})-V^j(t_{m-1})|\tau+C\tau^2. 
\end{multline*}
Since $\|W^i(t)-W^i(t_{m-1})\|\le C\tau$ and applying again Lemma \ref{lmm:normofrandomsum}, one has
\[
I_4\le C\sqrt{u}\tau+C\tau^2.
\]
The claim then follows by Gr\"onwall inequality as in Section \ref{sec:error}.
\end{proof}

\bibliographystyle{unsrt}
\bibliography{sdealg}

\end{document}